\documentclass{scrartcl}

\usepackage[utf8x]{inputenc}

\usepackage{amsthm}
\usepackage{amsmath}
\usepackage{mathtools}
\usepackage{amsfonts}
\usepackage[all]{xy}
\SelectTips{cm}{}

\usepackage{graphicx}
\usepackage[dvipsnames]{xcolor}
\usepackage{caption}

\usepackage{tikz}
\usetikzlibrary{calc}
\usetikzlibrary{cd} 

\usepackage[pagebackref,
    ,pdfborder={0 0 0}
  ,urlcolor=black,hypertexnames=false,
pdftitle={L2-Betti numbers and computability of reals}]{hyperref}

\usepackage{scrlayer-scrpage}

\automark[section]{section}
\usepackage{ifthen}
\deftriplepagestyle{myheadings}
             {\ifthenelse{\isodd{\value{page}}}{\leftmark}{\pagemark}}
             {}
             {\ifthenelse{\isodd{\value{page}}}{\pagemark}{\rightmark}}%
             {}{}{}%
\newtheoremstyle{ggt}{}{}{\itshape}{}{\sffamily\bfseries}{.}{ }{}
\newtheoremstyle{ggtdefinition}{}{}{}{}{\sffamily\bfseries}{.}{ }{}

\theoremstyle{ggt}
\newtheorem{thm}{Theorem}[section]
\newtheorem{lemma}[thm]{Lemma}
\newtheorem{cor}[thm]{Corollary}
\newtheorem{prop}[thm]{Proposition}

\theoremstyle{ggtdefinition}
\newtheorem{defi}[thm]{Definition}
\newtheorem{example}[thm]{Example}
\newtheorem{question}[thm]{Question}

\newtheorem{remark}[thm]{Remark}
\newtheorem*{ack}{Acknowledgements}

\DeclareMathOperator*{\colim}{colim}
\DeclareMathOperator{\leanop}{\mathsf{Lean}}
\def\lean{\ensuremath{\leanop}}

\newcommand{\leaninline}[1]{\lstinline!#1!}

\DeclareMathOperator{\tr}{tr}
\newcommand{\trC}[1]{\ensuremath{\tr_{\C #1}}}
\newcommand{\trCG}{\ensuremath{\trC{G}}}

\newcommand{\RC}{\ensuremath{\mathrm{RC}}}
\newcommand{\LC}{\ensuremath{\mathrm{LC}}}
\newcommand{\EC}{\ensuremath{\mathrm{EC}}}
  	 
\def\N{\mathbb{N}}
\def\Z{\mathbb{Z}}
\def\Q{\mathbb{Q}}
\def\R{\mathbb{R}}
\def\C{\mathbb{C}}

\DeclareMathOperator{\map}{map}
\DeclareMathOperator{\im}{im}

\def\args{\,\cdot\,}
\def\exi#1{\exists_{#1}\quad}
\def\fa#1{\forall_{#1}\quad}
\def\actson{\curvearrowright}

\def\vN#1{\mathcal{R}#1}
\def\ltb#1{b^{(2)}_{#1}}
\def\ltm#1#2{b^{(2)}(#1;#2)}
\def\ltt{\varrho^{(2)}}
\def\ltch#1#2{C^{(2)}_{#1}(#2)}

\def\ltmul#1{R_{#1}^{(2)}} 

\DeclareMathOperator{\Wr}{Wr} 
\newcommand{\ZPinv}{\ensuremath{\Z[P^{-1}]}} 

\def\genrel#1#2{%
  \langle #1 \mid #2\rangle}

\def\qand{\quad\text{and}\quad}
\usepackage{amssymb}

\usepackage{caption}

\makeatletter
\def\blfootnote{\xdef\@thefnmark{}\@footnotetext}
\makeatother
\def\draftinfo{}
\date{\today%
  \protect\blfootnote{\copyright{\ C.~L\"oh, M.~Uschold~2022}. 
    This work was supported by the CRC~1085 \emph{Higher Invariants} 
    (Universit\"at Regensburg, funded by the DFG) and is partially
    based on MU's MSc project. 
    \\
    MSC~2020 classification: 03D78, 20F10, 20C07, 20J06, 68V20
    \draftinfo}}

\allowdisplaybreaks[1]
  
\begin{document}

\title{$L^2$-Betti numbers and computability of reals}
\author{Clara L\"oh and Matthias Uschold}

\maketitle

\thispagestyle{empty}

\begin{abstract}
  We study the computability degree of real numbers
  arising as $L^2$-Betti numbers or $L^2$-torsion of
  groups, parametrised over the Turing degree of the
  word problem.
\end{abstract}

\section{Introduction}

A real number~$r$ is \emph{computable} if there exists a computable
sequence of rational numbers that converges to~$r$ in a computably
controlled way. Similarly, one obtains notions of right- and
left-computability, as well as versions that are parametrised over
Turing degrees (Section~\ref{subsec:comp}).

Several real-valued invariants from group theory and geometric
topology are known to lead to values with a computable structure.
In particular, such results give rise to corresponding non-realisability
results: sufficiently non-computable values cannot occur.

For example, for each recursively presented group, all values of
stable commutator length are right-computable; conversely, every
right-computable real can be realised as the stable commutator length of
some recursively presented group~\cite{heuer_sclrc}.  Another example
is that the simplicial volume of each oriented closed connected
manifold is a right-computable real number~\cite{heuerloeh_trans}. The
intrinsic computable structure of values also appears in other
places: for instance, the set of mapping degrees between two oriented
closed connected manifolds is recursively enumerable, which leads to
examples of sets of integers that cannot be realised as sets
of mapping degrees (Appendix~\ref{appx:deg}).

In the present article, we focus on the values of $L^2$-Betti numbers
and $L^2$-torsion arising from groups. $L^2$-Betti numbers can be described as
limits of characteristic sequences, whose elements are traces of
powers of matrices over the group ring
(Section~\ref{subsec:charseq}). Computing traces of such powers
involves determining specific coefficients of elements in the group
ring, and thus requires solving the word problem.

\subsection{Main results}

A straightforward spectral estimate shows that the right-computability
degree of $L^2$-Betti numbers is bounded from above by the Turing
degree of the word problem of the underlying group.

\begin{thm}[Theorem~\ref{thm:RC}]\label{thm:main:RC}
  Let $G$ be a finitely generated group with word problem of degree at
  most~$\mathbf{a}$. Moreover, let $m,n\in \N$ and $A\in M_{m\times
    n}(\Z G)$. Then, the $L^2$-Betti number~$\dim_{\vN G} \ker
  (\ltmul{A})$ is $\mathbf{a}$-right-computable.
\end{thm}

Left-computability requires additional control on the spectrum
(Theorem~\ref{thm:LC}).

\begin{thm}[Theorem~\ref{thm:L2-computable-det-class}]\label{thm:main:L2-computable-det-class}
  Let $\mathbf a$ be a Turing degree. There is an algorithm of Turing
  degree~$\mathbf a$ that, 
  \begin{itemize}
  	\item given a finitely generated group~$G$, given by a finite generating
  	set~$S$,
  	\item an algorithm of Turing degree~$\mathbf a$, solving the word
  	problem of~$G$,
  	\item a matrix $A\in M_{n\times n}(\Z S^*)$ whose image in 
  	$M_{n\times n}(\Z G)$ is of determinant class, and
  	\item a rational number, testifying that~$A$ is of 
  	  determinant class (see Lemma~\ref{lemma:certify-DCC}),
        \item
          computes a sequence $\N \to \Q$ effectively converging to~$\dim_{\vN G} \ker (\ltmul{A})$.
  \end{itemize}
  In particular, $\dim_{\vN G} \ker (\ltmul{A})$ is $\mathbf{a}$-computable.
\end{thm}

In the presence of positive Novikov--Shubin invariants, 
also the values of $L^2$-torsion are computable reals:

\begin{thm}[Theorem~\ref{thm:torsion}]\label{thm:main:torsion}
  Let $G$ be a finitely generated group with word problem
  of Turing degree at most~$\mathbf{a}$. Let $X$ be a finite free $G$-CW-complex
  all of whose $L^2$-Betti numbers are zero and all of whose
  Novikov--Shubin invariants are positive. Then the
  $L^2$-torsion~$\ltt(G \actson X)$ is $\mathbf{a}$-computable. 
\end{thm}

In Turing degree~$\mathbf{0}$, the corresponding versions of
Theorem~\ref{thm:main:RC} and
Theorem~\ref{thm:main:L2-computable-det-class} had already been
established by Groth~\cite{Groth2012}.

Spectral estimates as in the proofs of Theorem~\ref{thm:main:RC} and
Theorem~\ref{thm:main:L2-computable-det-class} also lead to a
quantitative version of L\"uck's approximation theorem
(Proposition~\ref{prop:quantitative-Lueck}). This gives more 
explicit computability statements in the case of finitely presented
residually finite groups (Corollary~\ref{cor:use-Lueck-quantitative}).

Following work of Pichot, Schick, and \.Zuk~\cite{Pichot2015} for
the Turing degree~$\mathbf{0}$, we show a realisation result
for computable numbers parametrised over the Turing degree of the word
problem:

\begin{thm}[Corollary~\ref{cor:realisation}]\label{thm:main:realisation}
  Let $\mathbf{a}$ be a Turing degree. The set of $L^2$-Betti numbers
  arising from finitely generated groups of determinant class with
  word problem of degree at most~$\mathbf{a}$ is equal
  to the set of nonnegative, $\mathbf{a}$-computable real numbers.
\end{thm}

However, not every single group gives rise to $L^2$-Betti numbers
of the Turing degree of the word problem. 

\begin{remark}[]
	Let $G$~be a finitely generated, torsion-free, solvable
	group with unsolvable word problem (such groups exist, see Corollary~\ref{cor:torsfree-solv}). Because $G$ is solvable, it is in Linnell's class $\mathcal{C}$. 
	Since $G$ is in this class and torsion-free, all $L^2$-Betti numbers arising from~$G$
	are integral~\cite[Theorem~1.5, p.~564]{Linnell}, hence in particular effectively
	computable. But the word problem in~$G$ is \emph{not} of Turing degree~$\mathbf{0}$
	by assumption.	
\end{remark}

Ordinary Betti numbers of finitely presented groups can be computed by
an algorithm of Turing degree~$\mathbf 3$~\cite{nabutovsky_weinberger}.
With analogous arguments, we obtain for $L^2$-Betti numbers:

\begin{thm}[Theorem~\ref{thm:algoL2B}]\label{thm:main:algoL2B}
  Let $\mathbf{a}$ be a Turing degree. Then there exists an
  algorithm of Turing degree at most~$\mathbf a ^{(4)}$ that
  \begin{itemize}
  \item given a finitely generated group~$G$, a finite generating set~$S$, 
    and an algorithm of degree at most~$\mathbf a$ solving
    the word problem for~$G$ with respect to~$S$, and given~$k\in \N$,
  \item computes the binary expansion for~$\ltb k (G)$
    (or detects that this value is~$+\infty$).
  \end{itemize}
\end{thm}

We do not know whether the bound~``$\mathbf a ^{(4)}$'' is optimal.
The realisation result Theorem~\ref{thm:main:realisation} 
suggests 
that there might not be a version of Theorem~\ref{thm:main:algoL2B} with
a uniform finite Turing degree instead of~``$\mathbf a ^{(4)}$''. 

\subsection{\texorpdfstring{Related work on computability and $L^2$-invariants}{Related work on computability and L2-invariants}}

The most prominent open problem on the range of values of $L^2$-Betti
numbers is the Atiyah problem~\cite[Chapter~10]{lueck_l2}. One version
of this problem asks whether all $L^2$-Betti numbers arising from
torsion-free groups are integers. While many positive examples are
known, several generalised versions for groups with torsion have been
disproved. Austin showed in a non-constructive way that irrational
values occur~\cite[Corollary~1.2]{Austin2013}. Grabowski established
that all non-negative reals arise from finitely generated groups and
gave the first explicit examples with irrational
values~\cite[Theorem~1.1]{Grabowski2014}.  The following question
remains open: Are all $L^2$-Betti numbers over finitely generated
torsion-free groups integers?

In the direction of Theorem~\ref{thm:main:algoL2B}, it was already
known that $L^2$-Betti numbers of matrices/groups are not computable
from finite presentations: a concrete example comes from lamplighter
groups. For finitely presented groups~$G$ that contain the group~$\Z/2
\wr \Z$ as a subgroup, the problem of determining whether a given
matrix over~$\Z[G^3]$ leads to a trivial $L^2$-Betti number or not is
known to be undecidable~\cite[Theorem~1.1]{Grabowski2015}.

Less specifically, a simple witness construction argument shows that
$L^2$-Betti numbers and cost of finitely presented groups are not
computable from finite presentations~\cite[Remark~8.11]{ffflmm_finpres}.


\subsection{\texorpdfstring{Implementation in \lean}{Implementation in lean}}

We formalised the results and proofs on computability of values of
$L^2$-Betti numbers of groups in the proof assistant
\lean~\cite{Avigad2021}.  \lean\ is based on dependent type theory and
offers a library that covers a substantial part of undergraduate
mathematics~\cite{Community2020}.

On the one hand, the implementation in a proof assistant gives a
verification of correctness of the formalised proofs. On the other
hand, the implementation process can also lead to new mathematical
insights: in proof assistants such as \lean\ it is often easier to
model abstract concepts than concrete constructions. For example, it
is easier to formalise tracial algebras than the group ring and it is
easier to formalise a relative version of computability than absolute
computability. In this way, working with proof assistants encourages a
declarative and modular style of mathematics.

Our implementation is available online \cite{leanproject} and is explained in an 
earlier version of this article \cite[Section~9]{firstversion}. 

\subsection*{Organisation of this article}

We recall basics on $L^2$-Betti numbers in Section~\ref{sec:prelimL2}
and basics on computability in
Section~\ref{sec:prelimcomp}. Section~\ref{sec:RCLC} contains the
proofs of Theorem~\ref{thm:main:RC} and the characterisation of
left-computability. In Section~\ref{sec:detclass}, we consider the
determinant class case; in particular, we prove
Theorem~\ref{thm:main:L2-computable-det-class} and
Theorem~\ref{thm:main:realisation}.  The quantitative approximation
theorem is established in Section~\ref{sec:quantitative-Lueck}.  In
Section~\ref{sec:torsion}, we prove Theorem~\ref{thm:main:torsion} on
$L^2$-torsion.  Theorem~\ref{thm:main:algoL2B} on the computation of
$L^2$-Betti numbers from a description of the group is shown in
Section~\ref{sec:computation-L2-Betti-of-groups}.

\begin{ack}
  CL is grateful to Nicolaus Heuer for discussions on an
  earlier incarnation of these questions. 
  We would like to thank Wolfgang L\"uck and Thomas Schick
  for pointing us to the estimates for $L^2$-torsion.
  We appreciate the detailed report and suggestions by the
  first referee. We would like to thank Vasco Brattka for a discussion on this article.
  We are grateful to Francesco Fournier-Facio for pointing out to us the strategy 
  for proving Proposition~\ref{prop:uncount-torsfree-solv}.
\end{ack}


\section{\texorpdfstring{Preliminaries on $L^2$-Betti numbers}{Preliminaries on L2-Betti numbers}}
\label{sec:prelimL2}

Originally, $L^2$-Betti numbers were introduced by Atiyah
in a geometric setting by analytic means~\cite{Atiyah1976}.
By now, many extensions and descriptions are available
in view of the work of Eckmann, Dodziuk, Farber, and L\"uck.
In this article, we will mostly refer to the books
by Lück~\cite{lueck_l2} and Kammeyer~\cite{Kammeyer2019}.

In this section, let $G$ be a finitely generated group.

\subsection{\texorpdfstring{$L^2$-Betti numbers}{L2-Betti numbers}}

Topologically, $L^2$-Betti numbers can be defined as follows: Let $X$
be a proper $G$-CW-complex of finite type. Then, the $L^2$-Betti
numbers of~$X$ are the von Neumann-dimensions of the homology groups
of the $L^2$-completion of the cellular chain complex
of~$X$~\cite[Chapter~3.3]{Kammeyer2019}. We say that a real number is
an $L^2$-Betti number arising from~$G$ if there is a $G$-CW-complex
that has this number as one of its $L^2$-Betti numbers. Equivalently,
these numbers admit an algebraic
description~\cite[Lemma~10.5]{lueck_l2}\cite[Proposition~3.29]{Kammeyer2019}:

\begin{defi}[$L^2$-Betti numbers arising from a group]
  \label{def:L2-Betti-numbers}
  Let $b\in \R_{\ge 0}$. We say that $b$ is an \emph{$L^2$-Betti
    number arising from $G$} if there are~$n,m\in \N$ and a
  matrix~$A\in M_{m\times n} (\Z G)$ such that
  \[
  b = \dim_{\vN G} \ker (\ltmul{A}).
  \]
  We explain the occuring terms below:
  \begin{itemize}
  \item By~$\ell^2 G$, we denote the Hilbert space
    of square-summable complex sequences on~$G$, i.e.,  
    functions $a:G\to\C$ such that $\sum_{g\in G} |a_g|^2 < \infty$.
  \item By~$\ltmul{A}$, we denote the bounded linear map 
    $(\ell^2 G)^m \to (\ell^2 G)^n$ that is given by 
    right-multiplication with $A$, and $\ker (\ltmul{A})$
    denotes its kernel, which is a $G$-submodule
    of~$(\ell^2 G)^m$. 
  \item By~$\dim_{\vN G}$, we denote the \emph{von Neumann dimension}. One
    should keep in mind that this number is a priori a non-negative
    real number and not necessarily an integer. For details on
    the construction of the von Neumann dimension, we
    refer to the literature~\cite[Chapter~1.1]{lueck_l2}\cite[Chapter~2.3]{Kammeyer2019}.
  \end{itemize}
  We abbreviate~$\ltm AG \coloneqq  \dim_{\vN G} \ker (\ltmul{A})$.
\end{defi}

\begin{remark}[self-adjointness]
  In Definition~\ref{def:L2-Betti-numbers}, we can equivalently also
  demand that $m=n$ and that $A$ is self-adjoint, i.e., that $A^* = A$
  (because $A$ and $AA^*$ have the same kernel). Here, the involution
  $\cdot^*$ is given by taking the transpose of the matrix and
  applying the elementwise involution $\sum_{g\in G} a_g\cdot g
  \mapsto \sum_{g\in G} a_g\cdot g^{-1}$ on~$\Z G$.
\end{remark}

Special $L^2$-Betti numbers arising from a group~$G$ are the
$L^2$-Betti numbers~$\ltb *(G)$ of the group~$G$ itself~\cite[Chapter~1 and~6.5]{lueck_l2}:
If $G$ is a group that admits a classifying space~$X$ of finite type
and $k \in \N$, then the $k$-th $L^2$-Betti number of~$G$ is equal
to~$\ltm{\Delta_k}G$, where $\Delta_k$ is the cellular Laplacian
in degree~$k$ of the universal covering of~$X$.

\subsection{Spectral measures}

We recall basic properties of spectral measures. The spectral
measure of a self-adjoint matrix over the group ring is
characterised by the property in Proposition~\ref{prop:spec-meas-poly},
relating integration over the spectral measure to traces.

\begin{defi}[trace~{{\cite[Equation~3.170]{lueck_l2}}}]
  \label{def:trace}
  Let $n\in \N$. We define the \emph{trace} by
  \begin{align*}
    \trCG : M_{n\times n} (\C G) & \longrightarrow \C\\
    A &\longmapsto \sum_{i=1}^n (A_{ii})_{e},
  \end{align*}
  where $(A_{ii})_{e}$ is the contribution of the neutral element
  of~$G$ to the $i$-th diagonal element of~$A$.
\end{defi}

\begin{prop}[characterisation of the spectral
    measure~{{\cite[Definition~5.9]{Kammeyer2019}}}]
  \label{prop:spec-meas-poly}
  Let $n\in \N$ and $A\in M_{n\times n}(\Z G)$ be self-adjoint.
  Let $\|\ltmul A\|$ denote the operator norm of~$\ltmul A$.
  
  Then, the \emph{spectral measure of~$A$} is the unique
  measure~$\mu_A$ on the interval~$[0, \|\ltmul A\|]$ (with the Borel
  $\sigma$-algebra) such that: for all polynomials $p\in \R[x]$, we
  have
  \[
  \trCG\bigl(p(A)\bigr) = \int_0^{\|\ltmul A\|} p(x)\; d\mu_A(x). 
  \]
\end{prop}

Conveniently, we can express $L^2$-Betti numbers
via the spectral measure.

\begin{prop}[$L^2$-Betti numbers via spectral measure~\protect{\cite{Lueck1994}\cite[p.~97]{Kammeyer2019}}]
  \label{prop:spec-meas-kernel}
  Let  $n\in \N$
  and $A\in M_{n\times n}(\Z G)$ be self-adjoint.
  Then, we have
  \[
  \ltm A G = \mu_A \bigl(\{0\}\bigr).
  \]
\end{prop}

\subsection{Characteristic sequences}
\label{subsec:charseq}

As in the proofs of approximation theorems for $L^2$-Betti numbers, 
we will use characteristic sequences to approximate $L^2$-Betti
numbers. The characteristic sequences are defined in terms of the trace
on matrices over~$\C G$ (Definition~\ref{def:trace}).

\begin{defi}[characteristic sequence~{{\cite[Definition~3.171]{lueck_l2}}}]\label{def:charseq}
  Let $A\in M_{m\times n} (\C G)$ and~$K \ge
  \|\ltmul{A}\|$. Then, we define the \emph{characteristic
    sequence}~$c(A,K) \coloneqq  (c(A,K)_p)_{p \in \N}$ of $A$ by
  \[
  c(A, K)_p \coloneqq  \trCG \bigl((1-K^{-2} \cdot A A^*)^p\bigr)
  \in \R_{\ge 0}. 
  \]
\end{defi}

\begin{prop}[$L^2$-Betti numbers via characteristic
    sequences~{{\cite[Theorem~3.172 (1), (2)]{lueck_l2}}}]
  \label{prop:char-seq}
  Let $A\in M_{m\times n} (\C G)$ and $K\ge 0$ with $K \ge
  \|\ltmul{A}\|$. Then, the characteristic sequence $(c(A,K)_p)_{p\in
    \N}$ is a monotone decreasing sequence of non-negative real
  numbers satisfying
  \[
  \ltm A G = \lim_{p\to\infty} c(A,K)_p.
  \]
\end{prop}

This approximation can be proved via the spectral measure.
Moreover, the rate of convergence is controlled by the spectral  
measure (see Theorem~\ref{thm:LC}). 

\section{Preliminaries on computability}
\label{sec:prelimcomp}

We collect basic notions from computability such as Turing degrees,
the computability of reals and limits, how to represent elements of finitely
generated groups and matrices over the group ring, and the word
problem for finitely generated groups.

\subsection{Turing degrees}

We quickly recall the concept of Turing degrees. More details
can be found in the work of Simpson~\cite{Simpson1977}. Algorithms
always refer to algorithms in the sense of Turing machines (possibly
amended by an oracle).

\begin{defi}[Turing reducible, equivalent, degree]
  Let $f,g : \N \to \N$.
  \begin{itemize}
  \item 	
    We say that $f$ is \emph{Turing reducible} to~$g$,
    denoted~$f\le_T g$ if there is an algorithm that
    computes~$f$, given an algorithm, called an
    \emph{oracle}, that computes~$g$.
  \item 
    We say that $f$ is \emph{Turing equivalent} to~$g$ if
    $f\le_T g$ and~$g\le_T f$.
  \item The \emph{Turing degree} of~$f$, denoted~$\deg(f)$ is the set
    of all functions~$\N \to \N$ that are Turing equivalent
    to~$f$.
  \end{itemize}
  We introduce the same notions for subsets $A\subset \N$ by
  considering their characteristic functions $\chi_A : \N \to
  \{0,1\}$. 
\end{defi}

\begin{prop}[{{\cite[Proposition~1.2]{Simpson1977}}}]
  The following statements
  hold for Turing degrees of functions~$\N\to \N$ and subsets
  of~$\N$ (as functions~$\N\to \{0,1\}$).
  \begin{enumerate}
  \item There is an uncountable number of degrees.
  \item The relation~$\le_T$ induces a partial ordering
    on the set of degrees: For $f,g:\N\to\N$, we define
    $\deg(f)\le_T \deg(g)$ if~$f\le_T g$.
  \item The degree $\mathbf{0} \coloneqq  \deg (\emptyset)$ is the least
    degree under this partial ordering. It is the set of all
    computable functions. When dealing with subsets of~$\N$,
    the degree~$\mathbf{0}$ is the set of all decidable sets.
  \end{enumerate}
\end{prop}

For every Turing degree~$\mathbf{a}$, there is bigger Turing
degree~$\mathbf a^{(1)} \coloneqq  j(\mathbf{a})$, defined via the jump
operator~\cite[p.~634]{Simpson1977} (see also the proof of
Proposition~\ref{prop:jump}).

Using a standard encoding of~$\Q$ as a sequence over~$\N$, we may 
also speak of Turing degrees of functions~$\N \to \Q$ etc..

\subsection{Computability of real numbers}\label{subsec:comp}

We recall different notions of computability of real numbers,
parametrised by Turing degrees. The definitions in the computable
case can be found in the survey by Rettinger and Zheng~\cite{Rettinger-Zheng-compute-real}.

\begin{defi}[$\mathbf{a}$-computability]
  \label{def:EC}
  Let $\mathbf{a}$ be a Turing degree of functions~$\N\to\Q$ and let~$r\in \R$.
  \begin{itemize}
  \item The real number~$r$ is called \emph{$\mathbf{a}$-computable}
    if there is a sequence~$q:\N\to \Q$ of degree at most~$\mathbf{a}$
    such that
    \[
    \fa{n\in \N} \; |r-q_n| \le 2^{-n}.
    \] 
    In this sitation, we say that $(q_n)_{n\in \N}$ \emph{effectively converges} to~$r$.
    We denote the set of $\mathbf{a}$-computable real numbers
    by~$\EC_\mathbf{a}$.
  \item The real number~$r$ is called
    \emph{$\mathbf{a}$-left-computable}
    (resp.\ \emph{$\mathbf{a}$-right-computable}) if there is a
    sequence~$q:\N\to \Q$ of degree at most~$\mathbf{a}$ such that
    \[
    r = \lim_{n\to \infty} q_n
    \]
    and $q_n \le r$ (resp.\ $q_n \ge r$) for all~$n\in \N$.  We denote
    the set of $\mathbf{a}$-left-computable
    (resp.\ $\mathbf{a}$-right-computable) real numbers
    by~$\LC_\mathbf{a}$ (resp.~$\RC_\mathbf{a}$).
  \end{itemize}
  In the case $\mathbf{a} = \mathbf{0}$, we speak of \emph{(effective)
    computability} and \emph{left-/right-computability}, respectively.
\end{defi}

\begin{prop}[characterisation of computability]
  \label{prop:char-computability}
  Let $\mathbf{a}$ be a Turing degree of functions~$\N\to\Q$ and
  let $r\in \R$. The following are equivalent:
  \begin{enumerate}
  \item The number~$r$ is $\mathbf{a}$-computable.
  \item There are sequences~$q:\N\to\Q$ and~$\epsilon:\N\to\Q$ of
    degree at most $\mathbf{a}$ such that $\lim_{n\to\infty}
    \epsilon_n = 0$ and for all $n\in \N$, we have
    \[
    |r - q_n| \le \epsilon_n.
    \] 
  \item The Dedekind cut~$\{x\in \Q \mid x<r\}$
    is a set of degree at most~$\mathbf{a}$.		
  \item There exist~$k\in \Z$ and a subset~$A\subset \N_{>0}$ of
    degree at most~$\mathbf{a}$ such that
    \[
    r = k + \sum_{n\in A} 2^{-n}.
    \]
  \end{enumerate}
\end{prop}
\begin{proof}
	The equivalence of 1 and 2 is a straightforward argument.
	In the following, we show the equivalence of 1 and 4. The equivalence
	of 1 and 3 then follows in a similar fashion.
	
	If $r\in \Q$, then 1 holds (take a constant sequence) and so does 4 as 
	the binary expansion of $r$ is periodic in this case. 
	We therefore suppose that $r\in \R \backslash \Q$.
	
	Assume that $1$ holds, i.e., there exists a rational sequence
        $(q_n)_{n\in \N}$ of degree at most $\mathbf{a}$ such that
	\[
		|r - q_n | \le 2^{-n}
	\]
	for all $n\in \N$. We set $k \coloneqq  \lfloor r\rfloor$. We give a recursive
	algorithm of degree at most $\mathbf{a}$ which determines whether $n\in A$: Consider the number 
	\[
		r' \coloneqq  k + \Bigl( \sum_{i \in A \cap \{0,\dots, n-1\}} 2^{-i}\Bigr) +2^{-n}.
	\]
	Because $r \not\in \Q$, we have $r' \neq r$, hence there is
	$j\in \N$ such that $|r'-r| > 2^{-j}$. Because we
	have $|r - q_j | \le 2^{-j}$ by assumption, 
	we must have $q_j < r'$ or $q_j > r'$.
	In the former case, $n\not\in A$, in the latter case, we have $n\in A$.
	
	Conversely, if 4 holds, then
	\[
		    \Bigl(k + \sum_{i\in A\cap\{0,\dots,n\}} 2^{-i}\Bigr)_{n\in \N}
	\]
	is a sequence of degree at most $\mathbf{a}$ witnessing that $r$ is
	$\mathbf{a}$-computable.
\end{proof}

\begin{remark}
  The statements in Proposition~\ref{prop:char-computability} are
  \emph{not} algorithmically equivalent, i.e., in general, there is no
  algorithm that, given one representation of an effectively
  computable number, produces one of the other representations.
  
  However, by adding~$1$ to the Turing degree, we can overcome this
  problem. This is illustrated by
  Lemma~\ref{lemma:effective-conv-to-binary}.
\end{remark}

\begin{lemma}[]
  \label{lemma:effective-conv-to-binary}
  Let $\mathbf{a}$ be a Turing degree. Then, there is 
  an algorithm of degree~$\mathbf{a}^{(1)}$ that
  \begin{itemize}
  \item given a sequence~$q: \N \to \Q$ of degree~$\mathbf{a}$ that
    effectively converges to its limit,
  \item outputs the binary expansion of~$\lim_{n\to\infty} q_n$.
  \end{itemize}
\end{lemma}

\begin{proof}
  Let $n + \sum_{i=0}^\infty a_i \cdot 2^{-i}$ be the binary expansion
  of the limit~$q^* \coloneqq  \lim_{n \to\infty} q_n$, i.e., $n\in \Z$ and
  $a_i\in \{0,1\}$; if this expansion is not unique, we prefer the
  finite one.  Suppose that $k\in \N$ and $n, a_0, \dots, a_{k-1}$ are
  already known.  Then, $a_k = 0$ if and only if there is~$l\in \N$
  such that
  \[
  q_l + 2^{-l} < n + \sum_{i=0}^{k-1} a_i \cdot 2^{-i} +
  1 \cdot 2^{-k},
  \]
  which can be decided by using an oracle for the halting problem,
  whence increasing the Turing degree by~$1$.	
\end{proof}

Similarly to the case of effective computability, we have:

\begin{prop}[characterisation of $\mathbf{a}$-left- and right-computability]
  \label{prop:char-left-right-computability}
  Let $\mathbf{a}$ be a Turing degree of functions~$\N\to\Q$ and
  let $r\in \R$. Then the following are equivalent:
  \begin{enumerate}
  \item The number $r$ is $\mathbf{a}$-left-computable
    (resp.\ $\mathbf{a}$-right-computable).
  \item There exists a monotonically increasing (resp.\ decreasing)
    sequence $q:\N\to\Q$ of degree at most $\mathbf{a}$ such that
    $\lim_{n\to\infty} q_n = r$.
  \item The Dedekind cut $\{x\in \Q \mid x<r\}$ (resp.\ $\{x\in \Q
    \mid x>r\}$) is the image of a function $\N\to \Q$ of degree at
    most $\mathbf{a}$.
  \item There exist $k\in \Z$ and a subset~$A\subset \N_{>0}$ that is
  	either empty or
    the image of a function~$\N\to \Q$ of degree at most~$\mathbf{a}$
    (resp.\ such that $\N_{>0} \backslash A$ is the image of
    a function~$\N\to \Q$ of degree at most~$\mathbf{a}$)
    such that
    \[
    r = k + \sum_{n\in A} 2^{-n}.
    \]
  \end{enumerate}
\end{prop}
\begin{proof}
  One can use the same arguments as for
  Proposition~\ref{prop:char-computability}.
\end{proof}

\begin{cor}[]
  \label{cor:EC-vs-LC-RC}
  Let $\mathbf{a}$ be a Turing degree of functions $\N\to\Q$.
  Then, we have
  \[
  \EC_\mathbf{a} = \LC_\mathbf{a} \cap \RC_\mathbf{a}.
  \]
\end{cor}

\begin{proof}
  This follows from Propositions~\ref{prop:char-computability}
  and~\ref{prop:char-left-right-computability}.
\end{proof}

\begin{example}\hfil 
  \begin{itemize}
  \item For every Turing degree~$\mathbf{a}$, the sets
    $\EC_\mathbf{a}$, $\LC_\mathbf{a}$ and~$\RC_\mathbf{a}$ are
    countably infinite.
  \item All algebraic numbers are $\mathbf{0}$-computable~\cite[p.~6]{Rettinger-Zheng-compute-real}.
  \item The transcendental numbers~$e$ and $\pi$ are $\mathbf{0}$-computable ~\cite[p.~5]{Rettinger-Zheng-compute-real}.
  \end{itemize}
\end{example}

\begin{prop}[]\label{prop:jump}
  For every Turing degree $\mathbf{a}$, we have
  $\LC_\mathbf{a} \backslash \EC_\mathbf{a} \neq \emptyset$
  and $\RC_\mathbf{a} \backslash \EC_\mathbf{a} \neq \emptyset$.
\end{prop}
\begin{proof}
  Let $f:\N\to\{0,1\}$ be a function such that $\mathbf{a} = \deg (f)$.
  We set
  \[
  x\coloneqq  \sum_{n\in \N} f^*(n) \cdot 2^{-n} \in [0,2],
  \]
  where $f^*:\N \to\{0,1\}$ is the \emph{jump} of~$f$. The jump of~$f$
  is defined as follows: Fix a Gödel numbering of all
  algorithms. Then, we set~$f^*(n) \coloneqq  1$ if the $n$-th algorithm
  halts on input~$n$ with oracle~$f$, and $f^*(n)\coloneqq 0$
  otherwise~\cite[p.~633]{Simpson1977}.
  
  We have $x\in \LC_\mathbf{a}$: Indeed, the set~$\{n\in \N \mid
  f^*(n) = 1\}$ is semi-decided with oracle~$f$ by the following
  algorithm: For~$n\in \N$, simulate the $n$-th algorithm on input~$n$
  with oracle~$f$. Once this simulation terminates, accept.
	
  Moreover, we have $x\not\in \EC_\mathbf{a}$: Assume for a
  contradiction that $x\in \EC_\mathbf{a}$. Then, by the
  characterisation of Proposition~\ref{prop:char-computability}, we
  have that the set~$\{n\in \N \mid f^*(n) = 1\}$, whence $f$, is
  $\mathbf{a}$-computable. Therefore, $f^*\le_T f$, contradicting an
  elementary property of the jump
  operator~\cite[Proposition~1.1(iv)]{Simpson1977}.

  Similarly, one can show that $-x\in \RC_\mathbf{a} \backslash
  \EC_\mathbf{a}$.
\end{proof}

\subsection{Computability of limits}

We investigate the computability of (iterated) limits.  We will use
these results in Section~\ref{sec:computation-L2-Betti-of-groups} to
prove Theorem~\ref{thm:main:algoL2B}.

\begin{lemma}[from limits to effective limits] 
  \label{lemma:algo-limits-to-effective-limits}
  Let $\mathbf{a}$ be a Turing degree.  There is an algorithm of
  degree at most~$\mathbf{a}^{(1)}$ that
  \begin{itemize}
  \item given a sequence~$q:\N\to\Q$ of 
    degree at most~$\mathbf{a}$ that is convergent in~$\R$,
  \item outputs a subsequence $r:\N\to\Q$ of $q$
    that converges effectively to~$q^* \coloneqq  \lim_{n\to\infty} q_n$,
    i.e., for all~$n\in \N$, we have
    \[
    |r_n - q^* | \le 2^{-n}.
    \]
  \end{itemize}
\end{lemma}

\begin{proof}
  Let $q:\N\to\Q$ be a convergent sequence of degree at
  most~$\mathbf{a}$ and let $q^* \coloneqq  \lim_{n\to\infty} q_n \in \R$ be
  the limit. 
  As $q$ converges, it is Cauchy, i.e., 
  \[
  \fa{m\in\N} \exi{N\in \N} \fa{n\ge N} |q_N - q_n| \le 2^{-(m+1)}.
  \]
  If we have a subsequence $(q_{N(m)})_{m\in \N}$ with Cauchy indices
  as above, we also have
  \[
  \fa{m\in\N}  |q_{N(m)}-q^*| \le 2^{-m},
  \]
  It thus remains to show that we can produce such a sequence
  algorithmically.  We consider the following algorithm~$T$ of degree
  at most~$\mathbf{a}$:
  \begin{itemize}
  \item On input $m,N\in \N$,
  \item for $n= N, N+1, \dots$, do
    \begin{itemize}
    \item if $|q_n-q_N| > 2^{-m}$, then halt;
    \item otherwise continue.
    \end{itemize}
  \end{itemize}
  We then consider the following algorithm, which uses the halting
  problem of~$T$ as oracle, and thus is of degree at
  most~$\mathbf{a}^{(1)}$:
  \begin{itemize}
  \item On input $m\in \N$,
  \item for $N=0,1,2, \dots$, do:
    \begin{itemize}
    \item use the oracle to decide if $T$ halts on input~$(m,N)$;
    \item if this is \emph{not} the case, halt and return~$q_N$.
    \end{itemize}
  \end{itemize}
  Then, by construction, this second algorithm halts because the 
  sequence~$q$ converges.
  This is an algorithm of degree at most~$\mathbf a^{(1)}$
  with the desired properties.
\end{proof}

\begin{lemma}[from double limits to single limits]
  \label{lemma:algo-double-limits-to-limits}
  Let $\mathbf{a}$ be a Turing degree. 
  There is an algorithm of degree at most~$\mathbf{a}^{(1)}$ that 
  \begin{itemize}
  \item given a function~$q: \N \times \N \to \Q$
    of degree at most~$\mathbf{a}$ such that 
    $
    \lim_{i\to\infty} \lim_{j\to\infty} q(i,j)
    $
    exists in~$\R$ or is equal to~$+\infty$,
    and additionally $\lim_{j\to\infty} q(i,j) \in \R$ for all $i\in \N$, 
  \item outputs a sequence~$r: \N \to \Q$ such that
    \[
    \lim_{n\to\infty} r_n = \lim_{i\to\infty} \lim_{j\to\infty} q(i,j).
    \]
  \end{itemize}
\end{lemma}
\begin{proof}
  Let $q: \N \times \N \to \Q$ as above. By applying the algorithm of
  Lemma~\ref{lemma:algo-limits-to-effective-limits} to the
  sequences~$q(i, \cdot)$ for all~$i\in \N$, we obtain a function~$\widetilde q:
  \N \times \N \to \Q$ of degree at most~$\mathbf{a}^{(1)}$. Then, the
  diagonalisation~$r(n) \coloneqq  \widetilde q (n,n)$ is also of degree at
  most~$\mathbf{a}^{(1)}$ and satisfies the desired property
  $
  \lim_{n\to\infty} r_n = \lim_{i\to\infty} \lim_{j\to\infty} q(i,j).
  $
\end{proof}

\begin{lemma}[divergence to $+\infty$]
  \label{lemma:algo:divergence}
  Let $\mathbf{a}$ be a Turing degree. There is an algorithm of degree
  at most~$\mathbf{a}^{(2)}$ that
  \begin{itemize}
  \item given a function~$q : \N \to \Q$ of degree at
    most~$\mathbf{a}$ such that $\lim_{k\to\infty} q_k$ exists in~$\R$
    or is equal to~$+\infty$,
  \item decides if $\lim_{k\to\infty} q_k = +\infty$.
  \end{itemize}
\end{lemma}
\begin{proof}
  Given the present assumptions, divergence to $+\infty$ is equivalent
  to (positive) unboundedness of the sequence, i.e., to 
  \[
  \fa{K\in \N} \exi{k\in \N} q_k > K,
  \]
  which can be checked by an algorithm of degree at most~$\mathbf{a}^{(2)}$.
\end{proof}

\begin{lemma}[iterated limits]
  \label{lemma:algo-iterated-limits-2}
  Let $\mathbf{a}$ be a Turing degree and $k\in \N$.  There is an
  algorithm of degree at most~$\mathbf{a}^{(k+1)}$ that
  \begin{itemize}
  \item given a function~$q : \N^k \to \Q$ of degree at most~$\mathbf{a}$ such that
    \[
    \lim_{i_1 \to\infty} \dotsm \lim_{i_k\to\infty} q(i_1, \dots,  i_k) 
    \]
    exists in~$\R$, or is equal to~$+\infty$, and all of the `inner' limits
    exist in $\R$,
  \item outputs either the binary expansion of this limit if it exists
    in~$\R$ or detects that the limit is equal to $+\infty$.
  \end{itemize}
\end{lemma}

\begin{proof}
  Using an iterated application of the algorithm of
  Lemma~\ref{lemma:algo-double-limits-to-limits} (applied to the
  degrees~$\mathbf{a}, \mathbf{a}^{(1)}, \dots, \mathbf{a}^{(k-2)}$), we
  obtain an algorithm of degree~$\mathbf{a}^{(k-1)}$ that produces a
  sequence~$r:\N \to \Q$ converging to the same limit.  We can then
  check with the algorithm from Lemma~\ref{lemma:algo:divergence}
  (applied to the degree~$\mathbf{a}^{(k-1)}$) if this limit is equal
  to~$+\infty$. This combined algorithm is therefore of degree at
  most~$\mathbf{a}^{(k-1+2)}$.  If the limit is finite, we apply the
  algorithms of Lemma~\ref{lemma:algo-limits-to-effective-limits}
  (applied to degree~$\mathbf{a}^{(k-1)}$) and
  Lemma~\ref{lemma:effective-conv-to-binary} (applied to
  degree~$\mathbf{a}^{(k)}$) to obtain the binary expansion of the
  limit. This part of the algorithm is therefore also of degree at
  most~$\mathbf{a}^{(k-1+2)}$.
\end{proof}

\subsection{Presenting elements and matrices over the group ring}
\label{sec:presenting-matrices}

We explain how matrices over group rings can be represented
in an algorithmic setting: 
Let $G$ be a finitely generated group with a finite generating 
set $S$. Let $R= \Z$ or $R=\Q$.
We then write~$R[S^*]$ for the set of all $R$-linear
combinations of words over~$S \cup S^{-1}$. Thus, elements
in~$R[S^*]$ represent elements in the group ring~$R[G]$.
Similarly, for~$m, n\in \N$, we write~$M_{m \times n}(R[S^*])$
for the set of $m \times n$-matrices over~$R[S^*]$. Elements
of~$M_{m \times n}(R[S^*])$ hence represent $m\times n$-matrices
over the group ring~$R[G]$.

\subsection{The word problem and computability of traces over group rings}

The word problem of a group is the following problem: For a word in 
terms of generators of the group, decide whether this word represents
the trivial element. 
More formally: 

\begin{defi}[word problem]
  Let $G$ be a finitely generated group with finite generating set~$S$.
  Denote by $S^*$ the set of all words in~$S$ and the (formal) inverses of~$S$.
  Then, we define the \emph{word problem} of~$G$ (with respect to~$S$) to be
  the set~$
  \{w\in S^* \mid w=e \text{ in } G\}.
  $
\end{defi}

\begin{remark}[]
  The degree of the word problem of a finitely generated group does
  \emph{not} depend on the chosen finite generating set~\cite[Lemma~2.2]{Miller1992}.

  We say that a finitely generated group \emph{has a solvable word
    problem} if its word problem is decidable.  This is the case if
  and only if the word problem is of degree~$\mathbf{0} = \deg
  (\emptyset)$.
\end{remark}

We observe that computing traces for matrices over the group ring is as
hard as solving the word problem. We will use this later to deduce
results on the computability degree of $L^2$-Betti numbers.

\begin{prop}
  \label{prop:word-problem-vs-trace}
  Let $G$ be a finitely generated group with finite generating
  set~$S\subset G$. Then, the following problem is Turing-equivalent to
  the word problem of~$G$:
  \begin{itemize}
  \item
    Given~$n\in\N$ and a matrix~$A\in M_{n\times n}(\Q[S^*])$,
  \item
    compute~$\trCG (A)\in \Q$.
  \end{itemize}
\end{prop}

\begin{proof}
  The word problem can be reduced to computing traces: For a 
  word $w\in S^*$ it is equivalent to decide whether $w=e$ in $G$
  and to decide whether $\trCG(1\cdot w) = 1$.
	
  Conversely, we can reduce the computation of traces to the word
  problem: Given $n\in\N$ and a matrix~$A\in M_{n\times
    n}(\Q [S^*])$, we proceed in the following steps:
  \begin{itemize}
  \item We collect all the words on the diagonal of~$A$ that have a
    non-trivial coefficient.
  \item For all of these words, we check if they represent $e\in G$
    using an oracle for the word problem of~$G$.
  \item We sum all the coefficients belonging to elements that 
    represent~$e\in G$.
  \item This sum then is the trace of~$A$.
    \qedhere
  \end{itemize}
\end{proof}

\section{\texorpdfstring{Right-/Left-computability of $L^2$-Betti numbers}%
                        {Right-/Left-computability of L2-Betti numbers}}
\label{sec:RCLC}

In this section, we present sufficient conditions for right- and
left-computability of $L^2$-Betti numbers over finitely generated
groups.

\subsection{Algorithmic computation of dimensions of kernels}

Right-computability is based on the following consequence of the
descriptions of $L^2$-Betti numbers through characteristic
sequences (Section~\ref{subsec:charseq}). We use presentations of matrices as described in
Section~\ref{sec:presenting-matrices}.

\begin{lemma}\label{lem:dimkeralgo}
  Let $\mathbf a$ be a Turing degree.
  There is an algorithm of Turing degree at most~$\mathbf a$
  that 
  \begin{itemize}
  \item given a finitely generated group~$G$ and a finite generating
    set~$S$ of~$G$, together with an algorithm of degree at
    most~$\mathbf{a}$ solving the word problem for~$G$ with respect to~$S$,
    given~$m,n \in \N$, and $A \in M_{m
      \times n}(\Z[S^*])$,
  \item computes a monotone decreasing sequence~$\N \to \Q$
    that converges to~$\ltm AG$.
  \end{itemize}
\end{lemma}

\begin{proof}
  We compute the sum of all absolute values of coefficients occurring in~$A$
  and call this number~$K\in \N$. This number satisfies
  $K \ge \|\ltmul A\|_\infty$. Hence, the characteristic sequence
  $\bigl(c(A,K)_p\bigr)_{p\in \N}$ is monotone decreasing and 
  converges to~$\ltm A G$ (Proposition~\ref{prop:char-seq}). 
 
  Moreover, the characteristic sequence is computable
  from the given data: For~$p \in \N$, we can compute~$(1-K^{-2} A A^*)^p$ as a matrix 
  in~$M_{m \times m}(\Q[S^*])$. We can then calculate the 
  trace by solving the word problem (Proposition~\ref{prop:word-problem-vs-trace}).
\end{proof}

\subsection{Right-computability}

The degree of right-computability is bounded by the degree of the word
problem.

\begin{thm}[right-computability]
  \label{thm:RC}
  Let $G$ be a finitely generated group with word problem of degree at
  most~$\mathbf{a}$. Moreover, let $m,n\in \N$ and $A\in M_{m\times
    n}(\Z G)$. Then, the $L^2$-Betti number $\ltm AG$ is
  $\mathbf{a}$-right-computable.
\end{thm}

\begin{proof}
  There exists a finite generating set~$S$ of~$G$, an algorithm of
  degree~$\mathbf a$ solving the word problem of~$G$ with respect
  to~$S$, and a representation of~$A$ in~$M_{m \times n}(\Z[S^*])$.
  Then, the algorithm of Lemma~\ref{lem:dimkeralgo} produces an
  monotone decreasing $\mathbf a$-computable sequence converging
  to~$\ltm AG$. Hence, $\ltm AG$ is $\mathbf a$-right-computable.
\end{proof}

\begin{cor}[right-computability, finitely presented case]
  \label{cor:RC-fin-pres}
  Let $G$ be a finitely \emph{presented} group, let $m,n\in
  \N$, and let $A\in M_{m\times n}(\Z G)$. Then, the $L^2$-Betti
  number~$\ltm AG$ is $\mathbf{1}$-right-computable.
\end{cor}

\begin{proof}
  For finitely presented groups, the word problem is not necessarily
  solvable but it is always recursively
  enumerable~\cite[Theorem~12.2]{Rotman1995} and thus of Turing degree
  at most~$\mathbf{1}$~\cite[p.~645]{Simpson1977}.  Hence, the claim
  follows from Theorem~\ref{thm:RC}.
\end{proof}

\begin{remark}[]
  The set of $L^2$-Betti numbers arising from all finitely presented
  groups contains all non-negative weakly computable numbers, i.e.,
  all numbers that can be written as~$x-y$ with~$x,y\in
  \LC_{\mathbf{0}}$ and~$x\ge y$. This can for instance be deduced
  from the examples and techniques of Pichot, Schick and
  \.Zuk~\cite[Remark~13.5 and~13.6, Section~11]{Pichot2015}.
	
  On the other hand, by Corollary~\ref{cor:RC-fin-pres}, this set is a
  subset of~$\RC_{\mathbf{1}} \cap \R_{\ge 0}$. It is an open
  question what the set of $L^2$-Betti numbers arising from all
  finitely presented groups exactly looks like.
\end{remark}

\subsection{Left-computability}

Left-computability holds under additional control on
the spectral measure.

\begin{thm}[left- and effective computability]
  \label{thm:LC}
  Let $G$ be a finitely generated group with word problem of degree
  $\mathbf{a}$. Moreover, let $m,n\in \N$ and let $A\in M_{m\times
    n}(\Z G)$.  The following are equivalent:
  \begin{enumerate}
  \item\label{i:LC}
    The $L^2$-Betti number~$\ltm AG$ 
    is $\mathbf{a}$-left-computable.
  \item\label{i:EC}
    The $L^2$-Betti number $\ltm AG$ 
    is $\mathbf{a}$-computable.
  \item\label{i:cseq}
    There exists a sequence~$\epsilon:\N\to \Q$ of degree at
    most~$\mathbf{a}$ such that $\lim_{k\to\infty} \epsilon_k = 0$ and
    for all $k\in \N_{>0}$, we have
    \[
    \fa{k\in \N} \mu_{AA^*}\bigl((0,1/k)\bigr) \le \epsilon_k.
    \]
  \end{enumerate}
\end{thm}

\begin{proof}
  The equivalence between items~\ref{i:LC} and~\ref{i:EC}
  follows from $\mathbf{a}$-right-computability (Theorem~\ref{thm:RC})
  and Corollary~\ref{cor:EC-vs-LC-RC}.
	
  We write $\Delta \coloneqq  AA^* \in M_{m\times m} (\Z G)$. Then $\Delta$ is
  self-adjoint and $\ker (\ltmul{\Delta}) = \ker
  (\ltmul{A})$. Therefore, it suffices to prove the claim for~$\ltm
  \Delta G$.  We view $\mu_\Delta$ as a measure on~$[0,d]$, where
  $K\ge \|\ltmul A\|_\infty$ and $d \coloneqq  K^2$.

  For the implication~\ref{i:LC} $\Rightarrow$~\ref{i:cseq}, let $\ltm
  \Delta G = \mu_\Delta(\{0\})$ be left-computable. By
  Proposition~\ref{prop:char-left-right-computability}, there exists a
  monotonically increasing sequence~$q:\N\to\Q$ of degree at
  most~$\mathbf{a}$ such that $\lim_{k\to\infty} q_k = \ltm \Delta G$.
  To show item~\ref{i:cseq} we construct a
  sequence~$(\epsilon_k)_{k\in\N}$ as follows: For~$k\in \N_{>0}$, we
  set
  \[
  p(k) \coloneqq  \max \Bigl\{p\in\N \Bigm| \Bigl( 1-\frac{1}{kd}\Bigr)^p \ge \frac 1 2\Bigr\}. 
  \]
  and 
  \begin{align*}
    \epsilon_k
    \coloneqq  2 \cdot \bigl( c(\Delta,d)_{p(k)} - q_k \bigr).
  \end{align*}
  The sequence~$p$ is computable. Since the characteristic
  sequence~$(c(\Delta,d)_n)_{n\in \N}$ is of degree at most~$\mathbf
  a$ (Lemma~\ref{lem:dimkeralgo}) and since $q$ is assumed to be
  $\mathbf a$-computable, also the sequence~$\epsilon$ is of degree at
  most~$\mathbf{a}$.
	
  Moreover, the sequence $(\epsilon_k)_{k\in \N}$ tends to zero for
  $k\to\infty$: Indeed, both the characteristic sequence
  $(c(\Delta,d)_k)_{k\in \N}$ and~$q$ tend to~$\ltm \Delta G$ and
  $\lim_{k\to\infty} p(k) = \infty$.

  Finally, we show for all~$k\in \N$ that $\mu_\Delta ((0,1/k)) \le
  \epsilon_k$. This estimate follows from:
  \begin{align*}
    c(\Delta,d)_{p(k)}
    & = 
    \trCG \Bigl(1- \frac{\Delta}{d} \Bigr)^{p(k)}
    = \int_0^d \Bigl(1- \frac{x}{d} \Bigr)^{p(k)}\; d\mu_\Delta(x) 
    &(\text{Proposition~\ref{prop:spec-meas-poly}})
    \\
    &\ge \int_{\{0\}} \Bigl(1- \frac{x}{d} \Bigr)^{p(k)}\; d\mu_\Delta(x)
    +
    \int_{(0,1/k)} \Bigl(1- \frac{x}{d} \Bigr)^{p(k)}\; d\mu_\Delta(x)
    \hspace{-2.5cm}
    \\
    &\ge \mu_\Delta\bigl(\{0\}\bigr)
    + \frac{1}{2} \cdot \mu_\Delta\bigl((0,1/k)\bigr)
    & (\text{choice of } p(k))
    \\
    &= \ltm \Delta G + \frac{1}{2} \cdot \mu_\Delta\bigl((0,1/k)\bigr)
    &(\text{Proposition~\ref{prop:spec-meas-kernel}})
    \\
    &\ge q_k + \frac{1}{2} \cdot \mu_\Delta\bigl((0,1/k)\bigr)
    & (q_k \text{ approx. from below})
  \end{align*}

  Conversely, we show the implication~\ref{i:cseq}
  $\Rightarrow$~\ref{i:LC}: Let $\epsilon:\N\to\Q$ be a sequence of
  degree at most~$\mathbf{a}$ as in item~\ref{i:cseq}. We will deduce
  left-computability. We denote
  \[
  q_k \coloneqq  c(\Delta,d)_{k^2}
  -\epsilon_k  - m \cdot \Big( 1 - \frac{1}{k\cdot d}\Big)^{k^2}.
  \]
  By assumption, $\lim_{k\to\infty} \epsilon_k = 0$. By elementary
  calculus,
  \[
  \lim_{k\to\infty} \Big( 1 - \frac{1}{k\cdot d}\Big)^{k^2} = 0.
  \]
  Hence, we obtain with Proposition~\ref{prop:char-seq} that
  \[\lim_{k\to\infty} q_k = \lim_{k\to\infty} c(\Delta,d)_{k^2}
  = \ltm \Delta G.
  \]
  The sequence~$(q_k)_{k\in \N}$ is of degree~$\mathbf{a}$ because its
  summands are of degree at most~$\mathbf{a}$; for the characteristic
  sequence, this follows from Lemma~\ref{lem:dimkeralgo}.

  The convergence of~$q$ is from below by the following calculation:
  \begin{align*}
    c(\Delta,d)_{k^2}
    & =\trCG \biggl(\Bigl(1-\frac{1}{d}\cdot \Delta\Bigr)^{k^2}\biggr)
    = \int_0^d \Big(1-\frac{x}{d}\Big)^{k^2}\;d\mu_\Delta
    & (\text{Proposition~\ref{prop:spec-meas-poly}})
    \\
    &\le \mu_\Delta\bigl(\{0\}\bigr)
    + \mu_\Delta\bigl((0, 1/k)\bigr)
    + \mu_\Delta\bigl([1/k, d]\bigr) \cdot \Bigl(1-\frac{1/k}{d}\Bigr)^{k^2}
    \hspace{-3cm}
    & (\text{monotonicity})
    \\
    &\le \mu_\Delta\bigl(\{0\}\bigr)
    + \epsilon_k
    +\mu_\Delta\bigl([0, d]\bigr) \cdot \Bigl(1-\frac{1}{kd}\Bigr)^{k^2}
    & (\text{assumption on }\epsilon_k)
    \\
    &= \ltm \Delta G
    + \epsilon_k
    +\mu_\Delta\bigl([0, d]\bigr) \cdot \Bigl(1-\frac{1}{kd}\Bigr)^{k^2}
    \hspace{-1cm}
    & (\text{Proposition~\ref{prop:spec-meas-kernel}})
    \\
    &= \ltm \Delta G
    + \epsilon_k
    +m \cdot \Bigl(1-\frac{1}{kd}\Bigr)^{k^2}
    & (\text{Proposition~\ref{prop:spec-meas-poly} with }p=1)
  \end{align*}
  Hence, $\ltm \Delta G$ is $\mathbf{a}$-left-computable.
\end{proof}


\section{The determinant class conjecture}
\label{sec:detclass}

We show that matrices of determinant class lead to effectively
computable $L^2$-Betti numbers (relative to the Turing degree of the
word problem) and the corresponding realisation result
(Theorems~\ref{thm:main:L2-computable-det-class}
and~\ref{thm:main:realisation}).

\subsection{The determinant class conjecture}

We first recall the definition of the Fuglede--Kadison determinant
and the determinant class conjecture. Details can be found in the
literature~\cite[Definition~1.3]{Schick2001}.

\begin{defi}[Fuglede--Kadison determinant]%
  Let $G$ be a group, $A\in M_{n\times n}(\Z G)$ be self-adjoint, and
  let $\mu_A$ be the spectral measure of~$A$.  Then, we define the
  \emph{Fuglede--Kadison determinant} of $A$ by
  \[
  \ln \det (A) \coloneqq  
  \begin{cases}		
    \int_{0^+}^\infty \log (x) \;d\mu_A(x) & \quad
    \text{if this integral converges}\\
	  -\infty & \quad\text{otherwise}.
  \end{cases} 
  \]
  Here, $\int_{0^+}^\infty$ denotes integration on the
  set~$(0,\infty)$.
\end{defi}

\begin{remark}[]
  There is only a convergence problem near~$0$ and no problem
  for~``$x\to\infty$'', as $\mu_A$ is supported on~$[0,\|\ltmul A\|]$.
\end{remark}

\begin{defi}[determinant class conjecture]
  \label{def:det-class}
  We say that a group~$G$ satisfies the \emph{determinant class
    conjecture} if for every~$n \in \N$ every self-adjoint element~$A
  \times M_{n\times n}(\Z G)$ satisfies~$ \ln \det A > -\infty.  $
\end{defi}

\begin{example}[]
  Sofic groups satisfy the determinant class
  conjecture~\cite[Theorem~5]{Elek2005}. For a sofic group~$G$, and a
  self-adjoint matrix~$A\in M_{n\times n}(\Z G)$, we even have~$\ln
  \det A \ge 0$.
\end{example}

It is not known whether all groups satisfy the determinant class conjecture.

\subsection{\texorpdfstring{Effective computability of $L^2$-Betti numbers}{Effective computability of L2-Betti numbers}}

From the property of being of determinant class,
we can deduce effective computability of the same degree
as the word problem. The case of degree~$\mathbf{0}$ was originally
proved by Groth~\cite[Theorem~6.12]{Groth2012}. 

First, note that the following holds.

\begin{lemma}[]
	\label{lemma:certify-DCC}
	Let $G$ be a finitely generated group 
	and let $A\in M_{n\times n}(\Z G)$ be
  self-adjoint and of determinant class.
  Then, there is $q\in \Q$ such that
  \[
  \int_{0^+}^1 \log (x) \; d\mu_A(x) \ge - q.
  \]
\end{lemma}

\begin{proof}
	Because $A$ is of determinant class, we have 
  in particular that
  $
  \int_{0^+}^1 \log (x) \; d\mu_A(x) > -\infty
  $. 
  Thus, there is a $q\in \Q_{\ge 0}$ satisfying the claim.
\end{proof}

We do not know if we can algorithmically compute such a rational number.

\begin{question}[]
	\label{question:dcc-computable}
	Let $S\subset G$ be a finite generating set.
	Is there an algorithm that, given a matrix~$A\in M_{n\times n}(\Z S^*)$
        whose image in~$M_{n\times n}(\Z G)$ is of determinant class, 
  	computes a~$q\in \Q_{\ge 0}$ that satisfies Lemma~\ref{lemma:certify-DCC}?
\end{question}

\begin{thm}
  \label{thm:L2-computable-det-class}
  Let $\mathbf a$ be a Turing degree. There is an algorithm of Turing
  degree~$\mathbf a$ that, 
  \begin{itemize}
  	\item given a finitely generated group~$G$, given by a finite generating
  	set~$S$,
  	\item an algorithm of Turing degree~$\mathbf a$, solving the word
  	problem of~$G$,
  	\item a matrix $A\in M_{n\times n}(\Z S^*)$ whose image in 
  	$M_{n\times n}(\Z G)$ is of determinant class, and
  	\item a rational number $q\in \Q_{\ge 0}$ as in Lemma~\ref{lemma:certify-DCC},
        \item
          computes a sequence $\N \to \Q$ that effectively converges to~$\ltm {A}{G}$.
  \end{itemize}
  In particular, $\ltm {A}{G}$ is $\mathbf{a}$-computable.
\end{thm}

\begin{proof}
  By the proof of Theorem~\ref{thm:LC}, it suffices to construct a sequence~$\epsilon:\N
  \to \Q$ of degree at most~$\mathbf{a}$ such that $\lim_{k\to\infty}
  \epsilon_k~=~0$ and for all $k\in \N_{>0}$, we have
  $
  \mu_A((0,1/k))~\le~\epsilon_k.
  $
  We set 
  \[\Bigl(\epsilon_k \coloneqq  \frac{q}{\lfloor \log k \rfloor}\Bigr)_{k\in \N_{\ge 2}},\]
  which is computable (so in particular of degree at most~$\mathbf{a}$). Moreover,  $\lim_{k\to\infty}
  \epsilon_k~=~0$ and for all $k\in \N_{\ge 2}$, we have
  \[
    - q
    \le \int_{0^+}^1 \log (x) \; d\mu_A(x)
    \le \int_{0^+}^{(1/k)^-} \log (x) \; d\mu_A(x)
    \le \log(1/k) \cdot \mu_A\bigl((0,1/k)\bigr),
  \]
  as desired.  
\end{proof}

\begin{remark}[sofic groups]
  For sofic groups, we can prove a result analogous to
  Theorem~\ref{thm:L2-computable-det-class} more directly, using the
  fact that the Cayley graphs of sofic groups can be approximated by
  finite graphs. The proof of Elek and Szab\'o that sofic groups
  satisfy the determinant class conjecture contains a result on the
  approximation of traces~\cite[Lemma~6.3]{Elek2005}. We can thus use
  an approach similar to the one of
  Section~\ref{sec:quantitative-Lueck} to obtain effective
  computability.
  
  Alternatively, one can use an estimate given in an article
  of Grabowski~\cite[Proposition~A1]{Grabowski2015}.
  
  One advantage of these approaches is that we obtain a slightly
  stronger statement: There is an algorithm, that, given $A$ and 
  an algorithm of degree~$\mathbf{a}$ solving the word problem in $G$, computes a sequence 
  witnessing that $\ltm{A}{G}$ is effectively computable of degree 
  at most~$\mathbf{a}$. This algorithm does not need  a rational number as in
   Lemma~\ref{lemma:certify-DCC} as an input. If Question~\ref{question:dcc-computable} has a 
   positive answer, we can also eliminate this dependence in 
   Theorem~\ref{thm:L2-computable-det-class}. 
\end{remark}

We do not expect that computability of $L^2$-Betti numbers can
be leveraged into a proof that the underlying group is of determinant
class.

\begin{remark}
  There is another perspective on Theorem~\ref{thm:L2-computable-det-class}:
  If $G$ is a finitely generated group with word problem of degree~$\mathbf a$
  and there is an $L^2$-Betti number arising from~$G$ that is \emph{not}
  an $\mathbf{a}$-computable real number, then Theorem~\ref{thm:L2-computable-det-class}
  shows that $G$ does \emph{not} satisfy the determinant class conjecture.
  
  However, it is perfectly possible that for all groups~$G$ with word
  problem of degree~$\mathbf a$ all $L^2$-Betti numbers arising
  from~$G$ are $\mathbf a$-computable $L^2$-Betti numbers --
  independently of the determinant class conjecture.
\end{remark}

\subsection{\texorpdfstring{Realisation of $L^2$-Betti numbers}{Realisation of L2-Betti numbers}}

Conversely, each non-negative $\mathbf{a}$-computable real arises as
an $L^2$-Betti number over a group with $\mathbf{a}$-solvable word
problem:

\begin{thm}
  \label{thm:EC-realised}
  Let $\mathbf{a}$ be a Turing degree and let $b\in \EC_\mathbf{a} \cap
  \R_{\ge 0}$. Then, there exists a finitely generated group of
  determinant class with word problem of degree at most~$\mathbf{a}$
  such that $b$ is an $L^2$-Betti number arising from~$G$.
\end{thm}

This theorem was first proved by Pichot, Schick and \.Zuk for
$\mathbf{a} = \mathbf{0}$~\cite[Remark~13.3]{Pichot2015}. We follow
their construction of such groups.

\begin{proof} 
  In analogy with the approach by Pichot, Schick, and
  \.Zuk~\cite[Proposition~11.4]{Pichot2015}, we have: 

  \begin{lemma}
    \label{lemma:EC-subset-U}
    Let $\mathbf{a}$ be a Turing degree and let $U\subset \R_{\ge 0}$ be a
    subset with the following properties:
    \begin{enumerate}
    \item\label{i:mult}
      The set~$U$ is closed under multiplication with and addition
      of non-negative rational numbers.
    \item\label{i:add}
      The set~$U$ is additively closed: If $r,s\in U$, then~$r+s\in U$.
    \item\label{i:binary}
      There are numbers~$a \in \Q_{\geq 0}$, $q\in \Q_{>0}$ and
      $d\in \N_{>0}$ such that for every strictly increasing 
      $\mathbf{a}$-computable sequence~$n:\N\to\N$, we have
      \[
      a+q\cdot \sum_{k=1}^\infty 2^{k-dn_k}\in U.
      \]
    \end{enumerate}
    Then, we have $\EC_\mathbf{a} \cap \R_{\ge 0} \subseteq U$.
  \end{lemma}

  For the proof of Theorem~\ref{thm:EC-realised}, we choose~$U$ to be
  the set of all $L^2$-Betti numbers arising from finitely generated
  groups of determinant class with word problem of degree at
  most~$\mathbf{a}$.

  The conditions~\ref{i:mult} and~\ref{i:add} of
  Lemma~\ref{lemma:EC-subset-U} are satisfied, because products of two
  groups from the given class preserve the desired properties and we
  can find the sums~\cite[Lemma~11.2]{Pichot2015} and
  products~\cite[Lemma~11.3]{Pichot2015} of the $L^2$-Betti numbers as
  $L^2$-Betti numbers of the product. Moreover, all non-negative
  rational numbers are $L^2$-Betti numbers of finite
  groups~\cite[Example~3.14]{Kammeyer2019} and finite groups are sofic
  and have even solvable word problem.

  As for condition~\ref{i:binary}, let $I \subset \N$ denote the image
  of~$n$.  Then, $I$ is an $\mathbf{a}$-decidable set. Pichot, Schick
  and \.Zuk construct a group~$G_I$ with the desired $L^2$-Betti
  number as follows: Let $\Gamma \coloneqq  \Z\wr\Z$, generated by two
  specific elements~$s_1, s_2$. We define
  \[
  G_I \coloneqq  \bigl((\Z/2)^{\oplus \Gamma} /V_{F_l, I } \bigr) \rtimes \Gamma.
  \]
  Here, we consider the action of $\Gamma$ by translation on 
  $(\Z/2)^{\oplus \Gamma} = (\Z/2)[\Gamma]$. Moreover, consider the basis 
  $\{\delta_g \mid g\in \Gamma\}$ of $(\Z/2)^{\oplus \Gamma}$, where 
  $\delta_g$ denotes the characteristic function of the set $\{g\}$.
  We define $u\coloneqq  \delta_{s_1^{-1}} + \delta_e + \delta_{s_1}$ and 
  $V_{F_l, I}$ to be the $(\Z/2)[\Gamma]$-submodule of $(\Z/2)[\Gamma]$
  spanned by all the elements of the form
  \[u-tu
  \]
  with $t\in \Lambda_I \coloneqq  \langle s_2^n s_1 s_2^{-n} \mid n\in
  I\rangle \subset \Gamma$.
  Pichot, Schick and \.Zuk show that
  there are $a$, $q$, and~$d$ (which do not depend on~$I$)
  such that 
  \[
  a+q'\cdot \sum_{k=1}^\infty 2^{k-d'(3\cdot n_k+2)}
  = a+(q'\cdot 2^{-2d'})\cdot \sum_{k=1}^\infty 2^{k-(3d')\cdot n_k}
  \]
  is an $L^2$-Betti number arising from~$G_I$~\cite[Theorem~10.1]{Pichot2015}.

  Hence, it remains to show that $G_I$ is finitely generated, of
  determinant class and has a word problem of degree at
  most~$\mathbf{a}$.

  The group~$G_I$ is generated by~$s_1$, $s_2$, and~$\delta_0$;
  moreover, it is the extension of an abelian group with
  quotient~$\Gamma$.  It is known that $G_I$ is
  sofic~\cite[Theorem~1(3)]{Elek2006}, hence satisfies the determinant
  class conjecture~\cite[Theorem~5]{Elek2005}.  

  Finally, the word problem of~$G_I$ is of degree at
  most~$\mathbf{a}$.  This follows by the same argument as in the
  case~$\mathbf{a} = \mathbf{0}$~\cite[Theorem~12.4]{Pichot2015}.

  Hence, Lemma~\ref{lemma:EC-subset-U} shows that the set~$U$ of all
  $L^2$-Betti numbers arising from finitely generated groups of
  determinant class with word problem of degree at most~$\mathbf{a}$
  contains~$\EC_\mathbf{a} \cap \R_{\ge 0}$.
\end{proof}

In combination with Theorem~\ref{thm:L2-computable-det-class}, we
obtain the following result:

\begin{cor}\label{cor:realisation}
  Let $\mathbf{a}$ be a Turing degree. The set of
  $L^2$-Betti numbers arising from finitely generated groups of
  determinant class with word problem of degree at most~$\mathbf{a}$
  is equal to~$\EC_\mathbf{a} \cap \R_{\ge 0}$.
  \hfill\qedsymbol
\end{cor}


\section{A quantitative version of L\"uck's approximation theorem}
\label{sec:quantitative-Lueck}

By L\"uck's approximation theorem, $L^2$-Betti 
numbers of finite type free $G$-CW-complexes can be approximated
by the ordinary Betti numbers of the quotient spaces associated
with residual chains. 

\begin{thm}[L\"uck's approximation theorem~\protect{\cite[Theorem~0.1]{Lueck1994}}]
  \label{thm:Lueck-approx}
  Let $X$ be a finite type free $G$-CW-complex.  Let $G$ be residually
  finite and let $(G_k)_{k\in\N}$ be a residual chain of~$G$. Then,
  for every $n\in \N$, we have
  \[
  b_n^{(2)}(G\curvearrowright X) = 
  \lim_{k\to\infty} \frac{b_n(G_k\backslash X)}{[G:G_k]}.
  \]
  Here, $b_n$ denotes the (ordinary) $n$-th Betti number of a
  CW-complex with $\C$-coefficients.
\end{thm}

The proof relies on the algebraic counterpart:

\begin{thm}[Lück's approximation theorem for matrices~{{%
        \cite[Chapter~13]{lueck_l2}\cite[p.~97]{Kammeyer2019}}}]
  Let $G$ be a countable residually
  finite group and let $A\in M_{n\times n}(\Z G)$ be self-adjoint. Let
  $(G_k)_{k\in\N}$ be a residual chain of~$G$.  For~$k\in \N$,
  let $A_k\in M_{n\times n}(\Z (G/G_k))$ be the image of~$A$ under
  the entrywise projection induced by the canonical projection~$G\to
  G/G_k$. Then, we have
  \[
  \ltm A G = \lim_{k\to\infty} \ltm {A_k}{G/G_k}.
  \]
\end{thm}

In this section, we quantify the rate of convergence in order to
establish effective computability of $L^2$-Betti numbers.
To do so, we need to control the residual chain. We formulate
this in terms of adapted sequences:

\begin{defi}[adapted sequence]
  \label{def:adapted-seq}
  Let $G$ be a countable residually finite group and let $A\in
  M_{n\times n}(\Z G)$ be self-adjoint. Let $(G_k)_{k\in\N}$ be a
  sequence of finite index, normal subgroups of~$G$.  For~$k\in
  \N$, let $A_k\in M_{n\times n}(\Z (G/G_k))$ be the image of~$A$
  under the canonical projection.  We say that the sequence~$(G_k)_{k\in\N}$ is
  \emph{adapted to~$A$} if for all~$k\in \N$,
  all entries of the diagonals of~$A^0, A^1, \dots, A^{k^2}$ have
  support in~$\{e\} \cup G \backslash G_k$. 
  Put differently: On these diagonals, all the coefficients belonging
  to~$G_k\backslash \{e\}$ vanish.
\end{defi}

\begin{lemma}[computability of an adapted sequence]
  \label{lemma:Lueck-computability-adapted-seq}
  There exists an algorithm that, 
  \begin{itemize}
  \item given 
    a residually finite group~$G$, given by a finite presentation~$\genrel SR$,
    and a matrix in~$M_{n\times n}(\Z G)$, represented by an
    element~$A\in M_{n\times n}(\Z [S^*])$,
  \item
    determines a sequence of finite index, normal subgroups~$(G_k)_{k\in\N}$
    of~$G$ that is adapted to~$A$ and outputs the
    sequence~$(\ltm {A_k}{G/G_k})_{k \in \N}$.

    Again, the~$A_k$ denote the images of~$A$ under the 
    projections~$\Z G \to \Z (G/G_k)$.
  \end{itemize}
	        
  In particular, every square matrix over~$\Z G$ admits an adapted
  sequence.
\end{lemma}

\begin{proof}
  With this kind of input, there exists an algorithm that solves the
  word problem in the finitely presented residually finite group~$G$
  with respect to~$S$ and such an algorithm can be determined
  algorithmically from the given presentation~$\genrel SR$
  of~$G$~\cite[Theorem~5.3]{Miller1992}.

  We consider the following algorithm: On input~$k\in \N$, we proceed
  in the steps:
  \begin{itemize}
  \item We compute the matrices~$A^0, A^1, \dots, A^{k^2}$ and the
    elements on their diagonals that are \emph{not} the neutral
    element~$e\in G$ and that have a non-trivial coefficient; this is
    possible through the solution of the word problem. We call these
    the \emph{non-trivial diagonal elements}.
  \item For~$p=1,2,3,\dots$, do the following:
    \begin{itemize}
    \item We enumerate all group homomorphisms~$G\to S_p$, where
      $S_p$ is the symmetric group on~$\{1,\dots, p\}$.
      This can be accomplished by enumerating all maps~$S\to S_p$,
      and then verifying whether the images of the relators
      are trivial.
    \item We check whether the non-trivial diagonal
      elements computed in the first step are \emph{not} mapped
      to~$e\in S_p$ by the group homomorphism~$G\to S_p$.
    \item Once such a group homomorphism~$f:G\to S_p$ is found, we
      continue with it as below.
    \end{itemize}
  \item We enumerate the subgroup~$H \coloneqq  \operatorname{im} f
    = \langle f(s) \mid s\in S\rangle_{S_p}$ of~$S_p$ and its
    composition table.
  \item We write $A_k \coloneqq  f_*(A) \in M_{n\times n} (\Z H)$.
  \item We fix an isomorphism $\Z H \cong \Z^{|H|}$ of $\Z$-modules 
    and rewrite the matrix~$A_k$ in the form~$A_k' \in  M_{n|H|\times n|H|} (\Z)$.
  \item We return~$(\dim_\Q \ker (\cdot A_k')) / |H|$.
  \end{itemize}
  It remains to show that this algorithm terminates and 
  returns the correct output.
  
  The non-trivial diagonal elements computed in the first step are
  finitely many elements of~$G\backslash \{e\}$.  Because $G$ is
  residually finite, there exists a group homomorphism from~$G$ to a
  finite group such that all these elements are \emph{not} in the
  kernel. Because every finite group is contained in some finite
  symmetric group, eventually such a homomorphism~$f: G \to S_p$ is
  found.
	
  Then, the subgroup~$\ker f \subset G$ is normal, of finite
  index, and satisfies the condition on the $k$-th element of a
  sequence adapted to~$A$ (Definition~\ref{def:adapted-seq}). Moreover, 
  $G / \ker f \cong \operatorname{im} f = H$.  Because $H$ is finite,
  we obtain~\cite[Example~2.39]{Kammeyer2019}
  \begin{align*}
    \ltm {A_k}{G/G_k}
    & = \ltm {A_k}{H}
    = \frac{\dim_\C \ker (\cdot A_k')}{|H|}
    = \frac{\dim_\Q \ker (\cdot A_k')}{|H|}.
  \end{align*}
  Therefore, the algorithm outputs a sequence with the claimed properties.
\end{proof}

\begin{remark}[]
  Note that unlike in the statement of Lück's approximation theorem,
  the adapted sequence fixed by the algorithm does not have to be 
  a residual chain.
\end{remark}

For adapted sequences, we can estimate the error of the normalised
Betti numbers of such sequences to the $L^2$-Betti number~$\ltm AG$.

\begin{prop}[Lück's approximation theorem, quantitative version]
  \label{prop:quantitative-Lueck}
  Let $G$ be a countable residually finite group.  Let $A\in
  M_{n\times n}(\Z G)$ be self-adjoint and let $(G_k)_{k\in \N}$ a
  sequence of finite index, normal subgroups that is adapted to $A$
  (see Definition~\ref{def:adapted-seq}). Then, for all~$k\in \N$, we
  have
  \[
  \bigl| \ltm A G - \ltm {A_k} {G/G_k}\bigr|
  \leq
  n \cdot \Bigl(1-\frac{1}{kd}\Bigr)^{k^2}
  + \frac{n\cdot \log d}{\log k}, 
  \] 
  where $d \coloneqq  \|\ltmul{A}\|$ is the operator norm of~$\ltmul{A}$.
\end{prop}

\begin{proof}
  Let $k\in \N$. As before, we denote by~$\mu_A$ the spectral measure
  of~$A$ and by~$\mu_{A_k}$ the spectral measure of~$A_k\in M_{n\times
    n}(\Z (G/G_k))$.  We can view both measures as measures on the
  same interval~$[0,d]$ as $\|\ltmul {A_k}\|$ and $\|\ltmul A\|$ are
  bounded by~$|A|_1$, i.e., by the sum of the absolute values of the
  coefficients in~$A$. Without loss of generality, we assume~$d \ge 1$.
	
  We consider the polynomials $p_k \coloneqq  (1- \frac{1}{d}\cdot x)^{k^2}$.
  Because the sequence~$(G_k)_{k\in \N}$ is adapted to~$A$, and $p_k$
  is a polynomial of degree~$k^2$, the coefficients on the diagonals
  of $p_k(A)$ and~$p_k(A_k)$ belonging to the neutral element in the
  group ring are identical. Hence we have
  $\trCG\bigl(p_k(A)\bigr)
  = \trC{(G/G_k)}\bigl(p_k(A_k)\bigr)
  .
  $
  Rewriting this equality using Proposition~\ref{prop:spec-meas-poly}  
  yields
  \begin{equation}
    \label{eq:Lueck-poly-integration-coincide}
    \int_0^d p_k(x) \;d\mu_A(x) = \int_0^d p_k(x) \; d\mu_{A_k}(x)
  \end{equation}
  
  A further ingredient for this proof are logarithmic bounds for
  spectral measures: for all~$\lambda \in (0,1)$, we have~\cite[Proof
    of Theorem~2.3.1]{Lueck1994}\cite[Proposition~5.18]{Kammeyer2019}
  \[\mu_{A_k} \bigl((0,\lambda)\bigr) \le \frac{n \cdot \log d}{|\log \lambda|}.
  \]
  Moreover, also $\mu_A$ has a logarithmic bound:
  For all~$\lambda \in (0,1)$, we have~\cite[Theorem~2.3(3)]{Lueck1994}
  \[\mu_{A} \bigl((0,\lambda)\bigr)
  \le \frac{n \cdot \log d}{|\log \lambda|}.
  \]  
  We now calculate for $k \ge 2$ (plus and minus-signs in the bounds of an integral
  suggest that we integrate over (half-)open intervals):
  \allowdisplaybreaks
  \begin{align*}
    \ltm {A_k}{G/G_k}
    \hspace{-2cm}
    \\
    & = \mu_{A_k}\bigl(\{0\}\bigr)
    & (\text{Proposition~\ref{prop:spec-meas-kernel}})\\
    &= \int_0^d p_k(x)\; d\mu_{A_k} - \int_{0^+}^d p_k(x) \; d\mu_{A_k}\\
    &= \int_0^d p_k(x)\; d\mu_A - \int_{0^+}^d p_k(x)\;d\mu_{A_k}&(\text{Equation~\ref{eq:Lueck-poly-integration-coincide}})\\
    &= \int_0^d p_k(x)\; d\mu_A - \int_{0^+}^{1/k^-} p_k(x)\;d\mu_{A_k}- \int_{1/k}^{d} p_k(x)\;d\mu_{A_k}
    \hspace{-1cm}
    \\
    &\ge \int_0^d p_k(x)\; d\mu_A - \mu_{A_k}((0,1/k)) - \int_{1/k}^d p_k(x)\; d\mu_{A_k}
    & (p_k(x) \le 1)\\
    &\ge \int_0^d p_k(x)\; d\mu_A - \frac{n\cdot \log d}{\log k} 
    - \int_{1/k}^d p_k(x) \; d\mu_{A_k}
    & (\text{log bound}, \lambda = 1/k)\\
    &\ge \int_0^d p_k(x)\; d\mu_A - \frac{n\cdot \log d}{\log k} 
    - \int_{1/k}^d p_k(1/k)\; d\mu_{A_k}
    &
    (p_k(x) \text{ is mon. decreasing})\\
    &\ge \int_0^d p_k(x)\; d\mu_A - \frac{n\cdot \log d}{\log k} 
    - \mu_{A_k} \bigl([1/k, d]\bigr) \cdot p_k(1/k)
    \hspace{-2cm}
    \\
    &\ge \int_0^d p_k(x)\; d\mu_A - \frac{n\cdot \log d}{\log k} 
    - \mu_{A_k} \bigl([0, d]\bigr) \cdot p_k(1/k)
    &(\text{monotonicity of }\mu_{A_k})\\
    &= \int_0^d p_k(x)\; d\mu_A - \frac{n\cdot \log d}{\log k} 
    - n \cdot \Bigl(1-\frac{1}{kd}\Bigr)^{k^2}
    &(\text{Proposition~\ref{prop:spec-meas-poly}}, p=1)\\
    &\ge \mu_A(\{0\}) -  \frac{n\cdot \log d}{\log k} 
    - n \cdot \Bigl(1-\frac{1}{kd}\Bigr)^{k^2}
    & (\text{monotonicity}, p_k(0) =1)\\
    &= \ltm AG -  \frac{n\cdot \log d}{\log k} 
    - n \cdot \Bigl(1-\frac{1}{kd}\Bigr)^{k^2}
    & (\text{Proposition~\ref{prop:spec-meas-kernel}})
  \end{align*}
  Therefore, we obtain
  \[
  \ltm A G - \ltm {A_k}{G/G_k}
  \leq 
  n \cdot \Bigl(1-\frac{1}{kd}\Bigr)^{k^2} + \frac{n\cdot \log d}{\log k}
  .
  \]
  Analogously, we can prove the lower bound by a similar calculation,
  exchanging the roles of $A$ and~$A_k$. Notice that in that argument,
  the logarithmic bound for~$A$ will then replace the one for~$A_k$.
\end{proof}

\begin{cor}
  \label{cor:use-Lueck-quantitative}
  There is an algorithm that
  \begin{itemize}
  \item given a residually finite group~$G$, given by a finite presentation~$\genrel SR$,
    and 
    a self-adjoint element in~$A \in M_{n\times n}(\Z G)$, given as an element in~$M_{n\times n}(\Z [S^*])$,
  \item
    outputs a computable sequence witnessing that $\ltm AG$ is effectively 
    computable, i.e., an algorithm for computing a sequence~$(q_k)_{k\in \N}$
    such that for all~$k\in \N$, we have
    \[\bigl|\ltm A G - q_k\bigr| \le 2^{-k}.\]
  \end{itemize}
  In particular, in this situation, the $L^2$-Betti number $\ltm AG$
  is effectively computable.
\end{cor}

\begin{proof}
  By Lemma~\ref{lemma:Lueck-computability-adapted-seq}, there is an
  algorithm that given these data, computes a sequence~$(G_k)_{k\in
    \N}$ adapted to~$A$ and outputs
  \[
  \bigl(\ltm {A_k}{G/G_k}\bigr)_{k\in \N}.
  \]
  By the quantitative version of Lück's approximation theorem
  (Proposition~\ref{prop:quantitative-Lueck}), we can estimate the
  difference between these numbers and the $L^2$-Betti number~$\ltm
  AG$.  These estimates are computable and tend to zero
  for~$k\to\infty$.  Thus, we can algorithmically choose a subsequence
  witnessing the effective computability of~$\ltm AG$.
\end{proof}

\begin{remark}[]
  In fact, the result of
  Corollary~\ref{cor:use-Lueck-quantitative} that $L^2$-Betti numbers of
  finitely presented, residually finite groups are effectively
  computable also follows from
  Theorem~\ref{thm:L2-computable-det-class}, as residually finite
  groups are sofic (and thus satisfy the determinant class
  conjecture). Moreover, finitely presented, residually finite groups
  have a solvable word problem~\cite[Theorem~5.3]{Miller1992}, i.e., a
  word problem of Turing degree~$\mathbf{0}$.
  
  The approach in this section shows that we can also use sequences as
  in Lück's approximation theorem to computably approximate
  $L^2$-Betti numbers.
\end{remark}

\section{\texorpdfstring{$L^2$-Torsion}{L2-Torsion}}
\label{sec:torsion}

The $L^2$-torsion is the torsion invariant associated with the
dimension~$\dim_{\vN G}$~\cite[Chapter~3]{lueck_l2}: The $L^2$-torsion
of an $L^2$-acyclic Hilbert chain complex is the negative of the
alternating sum of the Fuglede--Kadison determinants of the boundary
operators. In the presence of positive Novikov--Shubin invariants, one
can use the characteristic sequences (Definition~\ref{def:charseq}) to
conclude computability of values of $L^2$-torsion, similarly to
Theorem~\ref{thm:L2-computable-det-class}. We formulate and prove this
statement in the context of $L^2$-torsion of spaces.

\begin{thm}\label{thm:torsion}
  Let $G$ be a finitely generated group with word problem
  of Turing degree at most~$\mathbf{a}$. Let $X$ be a finite free $G$-CW-complex
  all of whose $L^2$-Betti numbers are zero and all of whose
  Novikov--Shubin invariants are positive. Then the
  $L^2$-torsion~$\ltt(G \actson X)$ is $\mathbf{a}$-computable. 
\end{thm}
\begin{proof}
  Under the given hypothesis, $X$ is
  det-$L^2$-acyclic~\cite[Theorem~3.93(7)]{lueck_l2}
  and so 
  \begin{align*}
    \ltt(G \actson X)
    & = \ltt\bigl(\ltch * {G \actson X} \bigr)
    \\
    & = -\frac12 \cdot
    \sum_{p=0}^N (-1)^p \cdot p \cdot
    \ln \det ( R^{(2)}_{\Delta_p});
  \end{align*}
  here, $N \coloneqq  \dim X$ and $\Delta_p \coloneqq  A_pA_p^* + A_{p+1}^*A_{p+1}$ is
  the cellular Laplacian~\cite[Lemma~3.30]{lueck_l2} associated with
  matrices~$A_p$ describing the cellular boundary operator of~$\ltch *
  {G \actson X}$ in degree~$p$ (with respect to cellular bases chosen
  in each degree) .

  The set of $\mathbf{a}$-computable reals is closed under
  addition and subtraction; this can be seen as in the
  case~$\mathbf{a} = \mathbf{0}$~\cite[Theorem~1.2(5)]{Zheng2003}.
  Therefore, it suffices
  to show that $\ln\det(R^{(2)}_{\Delta_p})$ is $\mathbf{a}$-computable
  for each~$p \in \{0,\dots, N\}$.

  So, let $p \in \{0,\dots, N\}$ and let $\alpha(R^{(2)}_{\Delta_p})$
  denote the $p$-th Novikov--Shubin invariant of~$\Delta_p$.
  As the Novikov--Shubin invariants of~$X$ are positive, also
  $\alpha(R^{(2)}_{\Delta_p})$ is positive~\cite[Lemma~2.17]{lueck_l2}.
  Thus, there exist~$M, \alpha \in \Q$ with
  \begin{align*}
    \| R_{\Delta_p}^{(2)}\|_\infty
    < M < \infty
    \qand
    0
    < \alpha
    < \alpha(R^{(2)}_{\Delta_p}).
  \end{align*}
  Let $c_p$ denote the number of $p$-cells of~$X$. 
  Then, the combinatorial description of $L^2$-torsion and the fact
  that $\dim_{\vN G} \ker(R^{(2)}_{\Delta_*}) = \ltb *(G \actson
  X) = 0$ shows that there exists a~$C \in \Q_{>0}$ such
  that~\cite[Theorem~3.172(5)]{lueck_l2}:
  \[ \fa{K \in \N}
  0 \leq -2 \cdot \ln \bigl(\det(R^{(2)}_{\Delta_p}) \bigr)
  + 2 \cdot c_p \cdot \ln M
  - \sum_{k=1}^K \frac1k \cdot c(\Delta_p,M)_k 
  \leq \frac C{K^\alpha}
  \]
  Because the word problem of~$G$ is of degree~$\mathbf{a}$,
  because logarithms of rationals are computable~\cite{Egbert1978},
  and by Proposition~\ref{prop:word-problem-vs-trace}, 
  the sequence
  \[ \biggl( c_p \cdot \ln M -\frac12 \cdot \sum_{k=1}^K \frac1k \cdot c(\Delta_p,M)_k 
     \biggr)_{K \in \N}
  \]
  is of degree~$\mathbf{a}$. 
  Moreover, the error estimate sequence~$(1/2 \cdot C/K^\alpha)_{K \in \N}$ is computable.
  Therefore, $\ln(\det(R^{(2)}_{\Delta_p}))$ is an $\mathbf{a}$-computable
  number (Proposition~\ref{prop:char-computability}).
\end{proof}

\section{\texorpdfstring{The computation of $L^2$-Betti numbers of groups}{The computation of L2-Betti numbers of groups}}
\label{sec:computation-L2-Betti-of-groups}

In analogy with work of Nabutovsky and Weinberger on the computation
of ordinary Betti numbers of finitely presented
groups~\cite{nabutovsky_weinberger}, we obtain:

\begin{thm}\label{thm:algoL2B}
  Let $\mathbf{a}$ be a Turing degree. Then there exists an
  algorithm of Turing degree at most~$\mathbf a^{(4)}$ that
  \begin{itemize}
  \item given a finitely generated group~$G$ and a finite generating set~$S$
    of~$G$, together with an algorithm of degree at most~$\mathbf a$ solving
    the word problem for~$G$ with respect to~$S$, and given~$k\in \N$,
  \item computes the binary expansion for~$\ltb k (G)$.
  \end{itemize}
\end{thm}

Here, ``computing the binary expansion for~$\ltb k (G)$''  
means the following: if $\ltb k (G)$ is finite, then the
binary expansion of this real number is computed; otherwise, the
value~$+\infty$ is returned.

Similarly to Nabutovsky and Weinberger, we rewrite~$\ltb k (G)$
as
\begin{align}
  \ltb k (G)
= \sup_{i \in \N} \inf_{j \in \N_{\geq i}}
\dim_{\vN G} \bigl(
\im (\varphi_{ij} \colon M_i \to M_j)
\bigr),
\label{eq:nw}
\end{align}
where $((M_i)_{i \in \N}, (\varphi_{ij})_{i \in \N, j \in \N_{\geq i}})$
is a directed family of finitely generated $\vN G$-modules that
we can compute from~$S$ and the given algorithm for the word
problem of~$G$ with respect to~$S$. The modules~$M_i$ arise as
homology of finitely presented complexes over~$G$ with respect to~$S$.
We introduce the following terminology:

\begin{defi}[finitely presented complex]\label{def:finprescomplex}
  Let $S$ be a finite generating set of a group~$G$ and let $k \in
  \N$. We write $[k+1] \coloneqq  \{0,\dots, k+1\}$.  A \emph{finitely
    presented complex of length~$k$ over~$G$ with respect to~$S$} is a
  pair~$((n_p)_{p \in [k+1]}, (A_p)_{p \in [k+1]})$, consisting of
  natural numbers~$n_0, \dots, n_{k+1}$ and matrices~$A_p \in M_{n_p
    \times n_{p-1}}(\Z[S^*])$ for all~$p \in [k+1]$ with the property
  that for all~$p \in \{1,\dots,k+1\}$ we have (as matrices
  over~$\Z[G]$)
  \[ A_p \cdot A_{p-1} = 0.
  \]
\end{defi}

In particular, a finitely presented complex over~$G$
describes a partial $\Z[G]$-chain complex with based
free $\Z[G]$-chain modules of finite rank.

\begin{proof}[Proof of Theorem~\ref{thm:algoL2B}]
  We proceed in the following steps: We describe an algorithm
  that produces chain complexes whose homology groups give a description
  of~$\ltb k (G)$ as in Equation~\eqref{eq:nw}. We then go into
  the details of the computations of~$\dim_{\vN G} (\im \varphi_{ij})$
  and how the $\sup$-$\inf$ affects the overall degree of compatibility.

  Let $T$ be the given algorithm of degree at
  most~$\mathbf a$ solving the word problem for~$G$ with respect
  to~$S$. 
  From the given finite generating set~$S$, we 
  obtain the sum of the augmentation maps:
  \[ \xymatrix{%
    \Z
    & \bigoplus_S \Z[G] \ar[l]_-{\varepsilon}.
  }
  \]
  Inductively, we can algorithmically extend this to a resolution
  up to degree~$k+1$: 

  We can recursively enumerate~$\ker \varepsilon$ using the
  algorithm~$T$. More precisely, we can algorithmically (from~$S$
  and~$T$) compute an ascending sequence~$(S_i)_{i\in \N}$ of finite
  sets and maps~$(\pi_i \colon S_i \to \ker \varepsilon)_{i \in \N}$
  with~$\bigcup_{i \in \N} \pi_i(S_i) = \ker \varepsilon$. We then
  consider for each~$i \in \N$ the complex 
  \[ \xymatrix{%
    \Z
    & \bigoplus_S \Z[G] \ar[l]_-{\varepsilon}
    & \bigoplus_{S_i} \Z[G] \ar[l]_-{\partial_{1,i}}
  }
  \]
  where $\partial_{1,i} (x \cdot 1) = \pi_i(x)$
  for all~$x \in S_i$.
  Inductively over the degrees and the previously constructed
  finitary intermediate steps, this leads to a sequence~$(C^{i}_*)_{i \in \N}$
  of partial $\Z[G]$-chain complexes (up to degree~$k+1$)
  with the following properties:
  \begin{itemize}
  \item For each~$i \in \N$ and each~$p \in \{0,\dots,k+1\}$,
    the chain module~$C^i_p = \bigoplus_{X_{p,i}} \Z[G]$
    is free over a finite set~$X_{p,i}$. 
  
    We have~$X_{p,i} \subset X_{p,j}$ for all~$p \in \{0,\dots,k+1\}$,
    $i \in \N$, $j \in \N_{\geq i}$. 
  \item
    The union/colimit~$\bigcup_{i \in \N} C_*^i$ is a partial
    resolution of~$\Z$ over~$\Z[G]$ up to degree~$k+1$.
  \item
    For each~$p \in \{0,\dots,k+1\}$,
    the sequence~$(X_{p,i})_{i \in \N}$ is algorithmically
    enumerable from~$S$ and~$T$. 

    Moreover, the corresponding matrices of the boundary operators
    of~$C_*^i$ are also algorithmically computable from~$S$ and~$T$;
    similarly, for the matrices that describe the inclusions~$C_*^i
    \hookrightarrow C_*^j$ for all~$j \in \N_{\geq i}$.
  \end{itemize}
  In other words: we can algorithmically compute from~$S$ and~$T$ a
  corresponding sequence of finitely presented complexes of length~$k$
  over~$G$ with respect to~$S$ (in the sense of
  Definition~\ref{def:finprescomplex}).

  For~$i \in \N$, we set~$D_*^i \coloneqq  \vN G \otimes_{\Z G} C_*^i$.
  We switch from~$\ell^2 G$ to the
  \emph{group von Neumann algebra}~$\vN G$~\cite[Definition~2.23]{Kammeyer2019} so that we can use
  the full power of the algebraic version of the
  theory~\cite[Chapter~6]{lueck_l2}\cite[Chapter~4.2]{Kammeyer2019}. 
  Because $\vN G \otimes_{\Z G} \args$ and homology are
  compatible with directed colimits, we obtain
  \[ H_k(G; \vN G)
  \cong_{\vN G}
  \colim_{i \in \N} H_k(D_*^i).
  \]
  Hence, we have~\cite[proof of Theorem~6.54, Equation~(6.55)]{lueck_l2}
  \begin{align*}
    \ltb k (G)
  = \sup_{i \in \N}
  \inf_{j \in \N_{\geq i}}
  \dim_{\vN G} \bigl( \im(\varphi_{ij} \colon H_k(D_*^i) \to H_k(D_*^j)) \bigr),
  \end{align*}
  where $\varphi_{ij} \colon H_k(D_*^i) \to H_k(D_*^j)$ is
  the $\vN G$-map induced by the inclusion~$C_*^i \hookrightarrow C_*^j$.
  Given~$i$, the sequence~$(\dim \im (\varphi_{ij}))_{j \in \N_{\geq i}}$
  is decreasing. Moreover, the arising sequence in~$i$ is increasing.
  Thus, 
  \begin{align}
    \label{eq:supinfdim}
    \ltb k (G)
  = \lim_{i \to \infty}
  \lim_{\N_{\geq i} \ni j \to \infty}
  \dim_{\vN G} \bigl( \im(\varphi_{ij} \colon H_k(D_*^i) \to H_k(D_*^j)) \bigr).
  \end{align}

  As second step, we consider the terms~$\dim_{\vN G} ( \im
  \varphi_{ij})$: For all~$i,j \in \N$ with~$j \geq i$, the
  dimension~$\dim_{\vN G} (\im \varphi_{ij})$ can be algorithmically
  computed through Lemma~\ref{lem:dimimalgo} below from the finite
  presentations of the complexes constructed above.  More precisely,
  from~$S$ and~$T$, we can algorithmically compute a function~$q
  \colon \N^3 \to \Q$ of degree at most~$\mathbf a$  with the following properties:
  \begin{itemize}
  \item For all~$i,j \in \N$ with~$j < i$, we have~$q(i,j,n) = 0$.
  \item For all~$i,j \in \N$ with~$j \geq i$, the sequence~$(q(i,j,n))_{n \in \N}$
    converges to~$\dim_{\vN G} \im \varphi_{ij}$.
  \end{itemize}
  
  Finally, we resolve the double limit:  
  Note that in Lemma~\ref{lemma:algo-iterated-limits-2}, only
  the outer limit can converge to~$+\infty$. We therefore
  obtain from Lemma~\ref{lemma:algo-iterated-limits-2} 
  that we can algorithmically
  compute from~$S$ and~$T$ the binary expansion of
  \[
  	\ltb k (G) = \lim_{i \to \infty}
  				\lim_{ j \to \infty} 
  				\lim_{n\to\infty} q(i,j,n)
  \]
  through an algorithm of Turing degree~$\mathbf a ^{(3 + 1)}$.
\end{proof}

\begin{remark}
  The argument of Nabutovsky and Weinberger for ordinary Betti
  numbers of groups is based on an algorithmic construction of
  a sequence of finite complexes that approximate a classifying
  space of the group in question. In the same way, one could
  also prove Theorem~\ref{thm:algoL2B}; however, in our equivariant
  setting, the algebraic approach seemed easier to describe.
\end{remark}

It remains to prove Lemma~\ref{lem:dimimalgo}. We first
rewrite the dimension of the image as sums and differences
of dimensions of kernels.

\begin{lemma}[dimension of the image of a map in homology]\label{lem:dimim}
  Let $G$ be a group and let $(D_*, \partial_*^D)$ be an $\vN G$-chain complex consisting
  of free $\vN G$-modules of finite rank.  Let $i_* : C_* \to D_*$
  be the inclusion of a subcomplex.
  Then, for all $k\in \N$, we have
  \[
  \dim_{\vN G} \im H_k(i_*) = \dim_{\vN G}\ker \partial_k^C + \dim_{\vN G} \ker \partial_{k+1}^D
  - \dim_{\vN G} \ker (\partial_{k+1}^D + i_k)
  \]
  where $\partial_{k+1}^D + i_k$ denotes the map
  \begin{align*}
    \partial_{k+1}^D+ i_k : D_{k+1} \oplus C_k & \to D_k\\
    (x,y) &\mapsto \partial_{k+1}^D x + y.
  \end{align*}
\end{lemma}

\begin{proof}
  The von Neumann dimension is additive for short exact
  sequences~\cite[Theorem~4.7(ii)]{Kammeyer2019}. In particular, we
  can compute the von Neumann dimension of quotients and for any
  $\vN G$-linear map, the dimensions of its kernel and its image
  add up to the dimension of its domain.  We have
  \[\im H_k(i_*) \cong (\ker \partial_k^C) / (\im \partial_{k+1}^D \cap C_k).\]
  Using additivity, we obtain
  \allowdisplaybreaks	
  \begin{align*}
    \dim_{\vN G} \im H_k(i_k)
    =\;& \dim_{\vN G}\ker \partial_k^C - \dim (\im \partial_{k+1}^D \cap C_k) \\
    =\;& \dim_{\vN G}\ker \partial_k^C - \dim_{\vN G} \im \partial_{k+1}^D - \dim_{\vN G} C_k + \dim_{\vN G} \im (\partial_{k+1}^D+ i_k)\\
    =\;& \dim_{\vN G}\ker \partial_k^C - (\dim_{\vN G} D_{k+1} - \dim_{\vN G} \ker \partial_{k+1}^D) \\&
    - \dim_{\vN G} C_k + \bigl(\dim_{\vN G} (D_{k+1}\oplus C_k) - \dim_{\vN G} \ker (\partial_{k+1}^D+ i_k)\bigr)\\
    =\;& \dim_{\vN G}\ker \partial_k^C + \dim_{\vN G} \ker \partial_{k+1}^D
    - \dim_{\vN G} \ker (\partial_{k+1}^D+ i_k),
  \end{align*}
  as claimed.
\end{proof}

\begin{lemma}[dimension of the image of a map in homology, algorithmically]\label{lem:dimimalgo}
  Let $\mathbf a$ be a Turing degree.
  There is an algorithm of Turing degree at most~$\mathbf a$
  that
  \begin{itemize}
  \item given a finitely generated group~$G$ and a finite
    generating set~$S$ of~$G$, together with an algorithm of degree at most~$\mathbf a$
    solving the word problem for~$G$ with respect to~$S$,
    given $k \in \N$,
    and finitely presented complexes~$(m_*,A_*)$ and $(n_* , B_*)$ of
    length~$k$ over~$G$ with respect to~$S$ with $m_j \leq n_j$ and $B_j|_{m_j
      \times m_{j-1}} = A_j$ for all~$j \in [k]$,
  \item computes a sequence~$\N \to \Q$ that converges
    to~$\dim_{\vN G} H_k(i_*)$,
    where $i_*$ denotes the inclusion between the two complexes.
  \end{itemize}
\end{lemma}
\begin{proof}
  By Lemma~\ref{lem:dimim}, we have 
  \begin{align*}
    \dim_{\vN G} H_k(i_*)
    &
  = \dim_{\vN G} \ker \ltmul{A_k} 
  + \dim_{\vN G} \ker \ltmul{B_{k+1}}
  - \dim_{\vN G} \ker \ltmul{E_k}
  \\
  & = \ltm {A_k}G + \ltm {B_{k+1}} G - \ltm{E_k} G, 
  \end{align*}
  where $E_k$ denotes the matrix obtained by stacking~$B_{k+1}$ on top
  of~$(I_{m_k} \mid 0)$. We can thus apply the algorithm from Lemma~\ref{lem:dimkeralgo}
  to the three matrices~$A_k$, $B_{k+1}$, $E_{k}$ (which are computable from
  the input) to compute this dimension.
\end{proof}

\appendix
\section{Sets of mapping degrees}\label{appx:deg}

Neofytidis, Shicheng Wang, and Zhongzi Wang~\cite{neofytidiswangwang}
study the question of which subsets of~$\Z$ (containing~$0$) are
realisable as sets of mapping degrees between oriented closed
connected manifolds. Comparing cardinalities shows that most subsets
of~$\Z$ are \emph{not} realisable in this way. Using a computability
argument, we can give ``explicit'' examples of non-realisable sets.

\begin{prop}\label{prop:degrecenum}
  Let $M$ and $N$ be oriented closed connected manifolds of the
  same dimension. Then the set
  \[ \deg(M,N)
  \coloneqq  \bigl\{ \deg f \bigm| f \in \map(M,N) \bigr\}
  \subset \Z
  \]
  of mapping degrees is recursively enumerable.
\end{prop}
\begin{proof}
  Let $n \coloneqq  \dim M = \dim N$. It is well-known that $M$ and $N$ are
  homotopy equivalent to finite simplicial
  complexes~\cite{siebenmann,kirbysiebenmann}.  Let $M \to |X|$ and $N
  \to |Y|$ be such homotopy equivalences and let $[X] \in
  H_n(|X|;\Z)$, $[Y] \in H_n(|Y|;\Z)$ be the corresponding images of
  the fundamental classes.  As mapping degrees can be described in
  terms of the effect on~$H_n(\args;\Z)$, we obtain that $\deg(M,N)$
  coincides with the set
  \[ \deg(X,Y)
  \coloneqq  \bigl\{ d \in \Z \bigm|
  \exi{f \in \map(|X|,|Y|)} H_n(f;\Z)[X] = d \cdot [Y]
  \bigr\}.
  \]
  Therefore, it suffices to show that $\deg(X,Y)$ is recursively
  enumerable.

  We use simplicial approximation to show that $\deg(X,Y)$ indeed
  is recursively enumerable: 
  for~$k \in \N$, let $X(k)$ the $k$-th iterated barycentric
  subdivision of $X$. In view of the simplicial approximation theorem,
  we have
  \[ \deg(X,Y) = \bigcup_{k \in \N} \deg^s\bigl(X(k), Y\bigr),
  \]
  where
  \[ \deg^s\bigl(X(k), Y\bigr)
  \coloneqq  \bigl\{ d \in \Z \bigm|
  \exi{f \in \map^s(X(k),Y)} H_n(|f|;\Z)[X] = d \cdot [Y]
  \bigr\}
  \]
  is the set of mapping degrees of \emph{simplicial} maps on 
  these subdivisions. 
  Recursive enumerability of~$\deg(X,Y)$ can thus be established as
  follows:
  \begin{itemize}
  \item For every~$k \in \N$, the finite simplicial
    complex~$X(k)$ can be algorithmically computed
    from~$X$.
  \item For every~$k \in \N$, the set~$\map^s(X(k),Y)$
    of all simplicial maps~$X(k) \to Y$ is recursive.
  \item
    For every~$k \in \N$ and each simplicial map~$f \colon X(k)
    \to Y$, we can algorithmically compute the unique~$d \in \Z$
    with $H_n(|f|;\Z)[X] = d \cdot [Y]$ via simplicial homology.
    \qedhere 
  \end{itemize}
\end{proof}

\begin{example}
  Let $A \subset \Z$ be a set that is \emph{not} recursively
  enumerable (e.g., the complement of a halting problem set).
  Then, $A_0 \coloneqq  A \cup \{0\}$ contains $0$, but $A_0$ is \emph{not}
  recursively enumerable and thus cannot be realised as a set
  of mapping degrees between oriented closed connected manifolds
  (Proposition~\ref{prop:degrecenum}).
\end{example}

\section{Torsion-free solvable groups}
\label{appx:torsfree-solv}

The goal of this appendix is to prove the following proposition.

\begin{prop}[]
	\label{prop:uncount-torsfree-solv}
	There exist uncountably many isomorphism types of $2$-generated, torsion-free, solvable
	groups.
\end{prop}

Because there are only countably many algorithms that could solve the word
problem and thus only countably many groups (up to isomorphism) with solvable word problem, we obtain:

\begin{cor}
	\label{cor:torsfree-solv}
	There exist finitely generated, torsion-free, solvable
	groups with unsolvable word problem.
\end{cor}

In order to prove Proposition~\ref{prop:uncount-torsfree-solv}, we start with 
the following class of groups. 

\begin{defi}[torsion-free abelian groups of rank one]
	Let $P$ be a set of prime numbers. We define the following (additive) 
	subgroup of~$\Q$: 
	\[
		\ZPinv \coloneqq \Big\{ \frac{a}{p_1\cdots p_n} \mid a\in \Z, n\in \N, p_1, \dots, p_n \in P\Big\} 
	\]
\end{defi}

However, unless $P$ is empty, $\ZPinv$ will not be finitely generated. We thus use the
following version of the embedding theorem by Neumann--Neumann.

\begin{thm}[{{\cite[Corollary~5.2 and Construction in Section~4]{NeumannNeumann-Embedding}}}]
	\label{thm:NeumannNeumann-embedding}
	Let $G$ be a countable, solvable group. Then, $G$ can be embedded in a solvable $2$-generator
	group~$H(G)$ that can be embedded in the group
	\[
		Q(G) \coloneqq (G\Wr \Z) \Wr \Z.
	\]
	Here, $G\Wr \Z \coloneqq (G^\Z) \rtimes \Z$ denotes the \emph{(unrestricted) wreath product}.
\end{thm}

Ultimately, we want to distinguish the isomorphism types of $H(\ZPinv)$. For this purpose,
we introduce the following invariant, which is an alteration of the notion of \emph{type}
for rank-one torsion-free abelian groups \cite[Chapter~VII]{Griffith-Infiniteabelian}.

\begin{defi}[]
	\label{def:tau-invariant}
	Let $G$ be a group. We define the following set of prime numbers associated to~$G$: 
	\[
		\tau(G) \coloneqq \big\{p \text{ prime} \mid \exists_{x\in G\backslash \{e\}}\; \forall_{k\in \N}\; \exists_{y\in G} \;\; y^{(p^k)} = x\big\}.
	\]
\end{defi}

This invariant has the following properties:

\begin{remark}[basic properties of $\tau(G)$]
	\label{rem:basic-prop-tau}
	Let $G$ be a group.
	
	\begin{enumerate}
		\item \label{rem:basic-prop-tau-subgroup}
		If $G$ is a subgroup of~$H$, then $\tau(G) \subseteq \tau(H)$.		
		\item \label{rem:basic-prop-tau-ZP}
		Let $P$ be a set of primes. Then, $\tau(\ZPinv) = P$. 
		\item \label{rem:basic-prop-tau-Wr}
		We have $\tau(G \Wr \Z) = \tau(G^\Z) = \tau (G)$. This can be proved as 
		follows: The inclusions `$\supseteq$' are clear by the first item. On the other hand,
		if $p\in \tau(G\Wr\Z)$, we have elements $x \in G\Wr \Z$ and
		$y(k) \in G\Wr\Z$ for all~$k\in \N$ as in Definition~\ref{def:tau-invariant}. Then, $x$ and thus all $y(k)$ 
		must lie in the kernel of $G \Wr \Z \twoheadrightarrow \Z$, thus in 
		$G^\Z$. We can then project $x$ and $y(k)$ to a non-trivial component of $x$
		to produce witnesses of the fact that $p\in \tau(G)$.
	\end{enumerate}
\end{remark}

\begin{proof}[Proof of Proposition~\ref{prop:uncount-torsfree-solv}]
	We consider the groups $H(\ZPinv)$ for all subsets $P$ of the prime numbers,
	where $H(\cdot)$ is taken as in Theorem~\ref{thm:NeumannNeumann-embedding}.
	These groups are $2$-generated and solvable. Moreover, because $\ZPinv$, hence
	$Q(\ZPinv) = (\ZPinv\Wr \Z) \Wr \Z$ is torsion-free, so is $H(\ZPinv)$.
	
	It remains to show that different sets of primes lead to non-isomorphic groups~$H(\ZPinv)$.
	By Remark~\ref{rem:basic-prop-tau}, we have
	\begin{align*}
		\tau (H(\ZPinv)) &\subseteq \tau (Q(\ZPinv)) & (\text{Remark~\ref{rem:basic-prop-tau}.\ref{rem:basic-prop-tau-subgroup}}) \\
		&= \tau ((\ZPinv\Wr \Z) \Wr \Z)\\
		&= \tau (\ZPinv\Wr \Z) & (\text{Remark~\ref{rem:basic-prop-tau}.\ref{rem:basic-prop-tau-Wr}})\\
		&= \tau (\ZPinv) & (\text{Remark~\ref{rem:basic-prop-tau}.\ref{rem:basic-prop-tau-Wr}})\\
		&= P &(\text{Remark~\ref{rem:basic-prop-tau}.\ref{rem:basic-prop-tau-ZP}}).
	\end{align*}
	On the other hand, $\ZPinv$ embeds into $H(\ZPinv)$, hence
	$P = \tau(\ZPinv) \subseteq \tau(H(\ZPinv))$. 
	Thus, we have $\tau(H(\ZPinv)) = P$, allowing us to recover $P$ from 
	$H(\ZPinv)$.
\end{proof}

{\small
  \bibliographystyle{alpha}
  \bibliography{bib_l2_comp}}

\vfill

\noindent
\emph{Clara L\"oh,\\ Matthias Uschold}\\[.5em]
  {\small
  \begin{tabular}{@{\qquad}l}
    Fakult\"at f\"ur Mathematik,
    Universit\"at Regensburg,
    93040 Regensburg\\
    \textsf{clara.loeh@mathematik.uni-r.de}, 
    \textsf{http://www.mathematik.uni-r.de/loeh}
    \\
    \textsf{matthias.uschold@mathematik.uni-r.de},
    \textsf{https://homepages.uni-regensburg.de/$\sim$usm34387/}
  \end{tabular}}

\end{document}